\ifdefined\synctex
\fi
\documentclass[12pt, reqno]{amsart}
\usepackage{amsmath}
\usepackage{amsfonts}
\usepackage{amssymb}
\usepackage{txfonts}
\usepackage{mathtools}
\usepackage{array}
\usepackage[shortlabels]{enumitem}
\usepackage{cite}
\usepackage[colorlinks=true,linkcolor=red,citecolor=red,urlcolor=red,final,hyperindex]{hyperref}

\setlist[enumerate]{leftmargin=2em,labelsep=.5em,topsep=.4\baselineskip,itemsep=.2\baselineskip}

\newcommand{\nc}{\newcommand}
\nc{\delete}[1]{}
\nc{\mlabel}[1]{\label{#1}}
\nc{\mcite}[1]{\cite{#1}}
\nc{\mref}[1]{\ref{#1}}
\nc{\meqref}[1]{\eqref{#1}}
\nc{\mbibitem}[1]{\bibitem{#1}}
\nc{\mrm}[1]{{\rm #1}}
\nc{\id}{\mrm{id}}
\nc{\Id}{\mrm{Id}}
\def\o{\otimes}

\newtheorem{thm}{Theorem}[section]
\newtheorem{lem}[thm]{Lemma}
\newtheorem{cor}[thm]{Corollary}
\newtheorem{pro}[thm]{Proposition}
\theoremstyle{definition}
\newtheorem{defi}[thm]{Definition}
\newtheorem{ex}[thm]{Example}
\newtheorem{rmk}[thm]{Remark}

\title[NS-Leibniz bialgebras]{NS-Leibniz bialgebras, the Yang--Baxter equation and relative Rota--Baxter operators}
\author[Li]{Bei Li}
\address{School of Mathematical Sciences, Qufu Normal University,  Qufu 273165, P. R. of China}
\email{libei9717@163.com}

\author[Wang]{Dingguo Wang \textsuperscript{*}}
\address{School of Mathematical Sciences, Qufu Normal University,  Qufu 273165, P. R. of China}
\email{dgwang@qfnu.edu.cn}

 \thanks{\textsuperscript{*}Corresponding author}
\date{}
\subjclass[2020]{17A30, 17A32, 17A36, 17B38.}
\keywords{NS-Leibniz algebra, Manin triple, Yang--Baxter equation,
relative Rota--Baxter operator, Nijenhuis operator}
\begin{document}
\baselineskip=15pt
\allowdisplaybreaks
\raggedbottom

\begin{abstract}
We introduce NS-Leibniz bialgebras and establish their equivalence
with matched pairs and Manin triples of NS-Leibniz algebras.
In the coboundary case, skew-symmetric solutions of the
NS-Leibniz Yang--Baxter equation yield NS-Leibniz bialgebras
and correspond to $\mathcal{O}$-operators for the dual regular representation.
We study quasi-triangular and factorizable structures, and relate
factorizable tensors to quadratic Rota--Baxter NS-Leibniz algebras
of any fixed nonzero weight. We also construct compatible
bialgebra structures from skew-adjoint Nijenhuis operators
commuting with the Rota--Baxter operator.
\end{abstract}
\maketitle
\markboth{NS-LEIBNIZ BIALGEBRAS}{NS-LEIBNIZ BIALGEBRAS}
\tableofcontents

\section{Introduction}
\mlabel{sec:introduction}
NS-Leibniz algebras decompose a Leibniz product into three
operations $\succ,\prec,\vee$. When $\vee=0$, the defining
identities reduce to those of a Leibniz-dendriform algebra.
In this paper, we introduce NS-Leibniz bialgebras and study
their representations, double constructions and associated
Yang--Baxter equation. When $\vee=0$ and $\Delta_\vee=0$, the definition recovers that of a Leibniz-dendriform bialgebra introduced by Sun and Guo~\mcite{SunGuo}.

\subsection{Leibniz algebras and NS-Leibniz algebras}
Leibniz algebras were first considered by Blokh \mcite{Blokh}
and were subsequently introduced by Loday in the study of
Lie algebra homology \mcite{Loday}. Their defining identity
retains the derivation property of a Lie bracket without
requiring skew-symmetry. Loday and Pirashvili developed their
representation theory and cohomology in \mcite{LodayPirashvili}.
A relative Rota--Baxter operator on a Leibniz algebra induces
two operations whose sum is a Leibniz product. The resulting
structure is called a Leibniz-dendriform algebra
\mcite{TangSheng}, or a pre-Leibniz algebra \mcite{DasPre}.
Its relation with symplectic structures and phase spaces was
studied in \mcite{TangXuSheng}.

Splitting an algebraic product into several operations also
occurs in dendriform and quadri-algebras
\mcite{LodayDialgebras,AguiarLoday}. Leroux introduced
NS-algebras in connection with Nijenhuis operators
\mcite{Leroux}; their relation with Nijenhuis algebras was
further studied by Lei and Guo \mcite{LeiGuo}.
Das and Guo \mcite{DasGuo} introduced NS-Leibniz algebras
from twisted relative Rota--Baxter operators. Such an algebra
has three products $\succ,\prec,\vee$, whose sum is Leibniz.
Setting $\vee=0$ recovers a Leibniz-dendriform algebra.
Ospel, Panaite and Vanhaecke \mcite{OPV} placed these
constructions in a general setting based on bimodules.

\subsection{Bialgebras and Yang--Baxter equations}
Lie bialgebras and the classical Yang--Baxter equation provide
a model for double constructions and coboundary bialgebras
\mcite{Drinfeld}. Analogous constructions have been developed
for associative algebras \mcite{Aguiar,Bai2010}, left-symmetric
algebras \mcite{Bai2008}, and Novikov algebras
\mcite{HongBaiGuo}. Tang and Sheng \mcite{TangSheng}
introduced Leibniz bialgebras and their classical Yang--Baxter
equation. Sun and Guo \mcite{SunGuo} developed the corresponding
theory for Leibniz-dendriform algebras.

Quasi-triangular and factorizable structures relate tensor
solutions to quadratic Rota--Baxter algebras. This relation
has been studied for Lie algebras \mcite{LangSheng},
associative algebras \mcite{ShengWang}, pre-Lie algebras
\mcite{WangBaiLiuSheng}, and Leibniz algebras
\mcite{BaiLiuShengTang}. These results suggest considering
both skew-symmetric tensors and tensors with invariant
symmetric parts in the NS-Leibniz setting.

\subsection{NS-Leibniz bialgebras and relative Rota--Baxter operators}
We first establish the representation theory of NS-Leibniz
algebras and determine the six actions of the dual
representation. The total dual actions are
$(L_\succ^*,-(L_\succ+R_\prec)^*)$.
When $\vee=0$ and the actions associated with $\vee$ vanish,
these structures reduce to the representations and dual
representations of Leibniz-dendriform algebras.

We define NS-Leibniz bialgebras by four coalgebra identities
and six compatibility conditions, with total coproduct
$\Delta=\Delta_\succ+\Delta_\prec+\Delta_\vee$.
We establish their equivalence with matched pairs and
Manin triples. In the coboundary case, substituting the
three coproducts into these identities leads to the
NS-Leibniz Yang--Baxter equation. Its skew-symmetric
solutions determine NS-Leibniz bialgebras.

We then investigate $\mathcal O$-operators, quasi-triangular
and factorizable bialgebras, and quadratic Rota--Baxter
algebras. When $\vee=0$ and $\Delta_\vee=0$, the bialgebra
structures reduce to Leibniz-dendriform bialgebras
\mcite{SunGuo}.
Finally, skew-adjoint Nijenhuis operators commuting with a
quadratic Rota--Baxter operator produce compatible families
of NS-Leibniz bialgebras. The Leibniz results in
\mcite{MondalSaha,MaSunZheng} provide the background for
this construction.

\subsection{Outline of the paper}
Section~\mref{sec:representations} concerns representations,
matched pairs and Manin triples of NS-Leibniz algebras.
Section~\mref{sec:manin-bialgebras} introduces NS-Leibniz
bialgebras, coboundary structures and the Yang--Baxter
equation, followed by $\mathcal O$-operators and their induced
splittings. Section~\mref{sec:quasitriangular} treats
quasi-triangular and factorizable NS-Leibniz bialgebras.
Section~\mref{sec:relative-rb} relates these structures to
relative Rota--Baxter operators and quadratic Rota--Baxter
NS-Leibniz algebras. Section~\mref{sec:nijenhuis} contains
the Nijenhuis constructions and compatible bialgebras.

\subsection{Conventions and notation}
Throughout the paper, $k$ is a field of characteristic zero,
and all vector spaces are finite-dimensional over $k$.
The identity map is denoted by $I$, and
$\tau(x\o y)=y\o x$ is the flip map.
We write $\langle\zeta,x\rangle$ for the natural pairing
between a vector space and its dual.

\section{Representations and matched pairs of NS-Leibniz algebras}
\mlabel{sec:representations}
\mlabel{sec:matched-pairs}
This section develops representations, dual representations
and matched pairs of NS-Leibniz algebras. We begin with
the underlying Leibniz algebra and its representations.
Recall that a left Leibniz algebra $(A,\circ )$ satisfies
\begin{equation}\mlabel{eq:leibniz}
 x\circ (y\circ z)=(x\circ y)\circ z+y\circ (x\circ z).
\end{equation}
A representation is a vector space $V$ with linear maps
$l_\circ,r_\circ:A\to\mathrm{End}(V)$ satisfying
\begin{equation}\mlabel{eq:leibniz-rep}
 l_\circ(x\circ y)=[l_\circ(x),l_\circ(y)],\quad
 r_\circ(x\circ y)=l_\circ(x)r_\circ(y)-r_\circ(y)l_\circ(x),\quad
 r_\circ(y)\bigl(r_\circ(x)+l_\circ(x)\bigr)=0.
\end{equation}
For a linear map $T:V\to W$, write $T^{\mathsf t}:W^*\to V^*$
for its ordinary transpose. For an action $a:A\to\mathrm{End}(V)$,
define $a^*:A\to\mathrm{End}(V^*)$ by
\[
 \langle a^*(x)\zeta,u\rangle=-\langle \zeta,a(x)u\rangle.
\]
Thus $a^*(x)=-a(x)^{\mathsf t}$. For an ordinary linear map,
we also use $T^*=T^{\mathsf t}$.
The corresponding dual representation is $(V^*,l_\circ^*,-l_\circ^*-r_\circ^*)$.

\begin{defi}\mlabel{def:ns}
An NS-Leibniz algebra is a vector space $A$ with three bilinear operations
$\succ,\prec,\vee$ satisfying, with $x\circ y=x\prec y+x\succ y+x\vee y$,
\begin{align}
 x\prec(y\circ z)&=(x\prec y)\prec z+y\succ(x\prec z),\mlabel{eq:ns1}\\
 x\succ(y\prec z)&=(x\succ y)\prec z+y\prec(x\circ z),\mlabel{eq:ns2}\\
 x\succ(y\succ z)&=(x\circ y)\succ z+y\succ(x\succ z),\mlabel{eq:ns3}\\
 x\succ(y\vee z)+x\vee(y\circ z)
 &=(x\vee y)\prec z+(x\circ y)\vee z
   +y\succ(x\vee z)+y\vee(x\circ z).\mlabel{eq:ns4}
\end{align}
\end{defi}
Summing \meqref{eq:ns1}--\meqref{eq:ns4} shows that $(A,\circ)$
is a Leibniz algebra. We use the conventions of \mcite{DasGuo,OPV}.
For each $\diamond\in\{\succ,\prec,\vee\}$, set
$L_\diamond(x)y=x\diamond y$ and $R_\diamond(x)y=y\diamond x$.
Also set $L_\circ=L_\succ+L_\prec+L_\vee$,
$R_\circ=R_\succ+R_\prec+R_\vee$, and $L_\star=L_\circ+R_\circ$.
For module actions introduced below, we use $l_\circ=l_\succ+l_\prec+l_\vee$,
$r_\circ=r_\succ+r_\prec+r_\vee$, $l_\star=l_\circ+r_\circ$,
$l_\odot=l_\succ+r_\prec$, and $r_\odot=r_\succ+l_\prec$.
If $\vee=0$, $A$ is a Leibniz-dendriform algebra.

When $\vee=0$, the fourth identity vanishes and
\meqref{eq:ns1}--\meqref{eq:ns3} are precisely the
Leibniz-dendriform identities \mcite{SunGuo}.
The relation with the subadjacent Leibniz representation
is described by the following proposition.

\begin{pro}\mlabel{prop:ns-leibniz-representation}
Let $A$ have three bilinear products $\succ,\prec,\vee$, and put
$\circ=\succ+\prec+\vee$.
Then $A$ is an NS-Leibniz algebra if and only if $(A,\circ)$
is a Leibniz algebra and $(A,L_\succ,R_\prec)$ is a
representation of $(A,\circ)$.
\end{pro}
\begin{proof}
The representation identities for $(L_\succ,R_\prec)$ are
\begin{equation}\mlabel{eq:ns-underlying-representation}
\begin{aligned}
 L_\succ(x\circ y)&=[L_\succ(x),L_\succ(y)],\\
 R_\prec(x\circ y)&=L_\succ(x)R_\prec(y)-R_\prec(y)L_\succ(x),\\
 R_\prec(y)\bigl(R_\prec(x)+L_\succ(x)\bigr)&=0.
\end{aligned}
\end{equation}
The first two equalities are \meqref{eq:ns3} and
\meqref{eq:ns2}. Together, \meqref{eq:ns1} and
\meqref{eq:ns2} imply the third equality. Conversely,
the last two equalities imply \meqref{eq:ns1}.
Thus \meqref{eq:ns1}--\meqref{eq:ns3} are equivalent to
\meqref{eq:ns-underlying-representation}.

Summing the four NS-Leibniz identities yields the Leibniz
identity for $\circ$. Conversely, expand that Leibniz identity
using $\circ=\succ+\prec+\vee$, and subtract
\meqref{eq:ns1}--\meqref{eq:ns3}. The remaining equality is
\meqref{eq:ns4}.
\end{proof}

\begin{ex}\mlabel{ex:one-dimensional-ns}
Let $A=ke$ and define
$e\succ e=e$, $e\prec e=0$, $e\vee e=-e$.
The total product is zero. Direct substitution verifies
\meqref{eq:ns1}--\meqref{eq:ns4}.
\end{ex}

\begin{defi}\mlabel{def:quadratic}
A quadratic NS-Leibniz algebra is an NS-Leibniz algebra $A$ with a
nondegenerate symmetric bilinear form $\omega$ satisfying
\begin{align}
\omega(x\prec y,z)&=\omega(x,y\circ z+z\circ y),\mlabel{eq:inv-p}\\
\omega(x\succ y,z)&=-\omega(y,x\circ z).\mlabel{eq:inv-s}
\end{align}
We write $\omega^\sharp:A\to A^*$ for the map defined by
$\langle\omega^\sharp(x),y\rangle=\omega(x,y)$.
\end{defi}

The symmetry of $\omega$ implies
\begin{align}
\omega(x\vee y,z)&=\omega(y,x\vee z),\mlabel{eq:inv-v-left}\\
\omega(x\vee y,z)&=-\omega(x,y\vee z+z\vee y).\mlabel{eq:inv-v-right}
\end{align}
Indeed, expanding $\vee=\circ-\prec-\succ$ yields
\[
\omega(x\vee y,z)=\omega(x\circ y,z)+\omega(y,x\circ z)
 -\omega(x,y\circ z+z\circ y).
\]
The right-hand side is symmetric in $y,z$, and its cyclic sum
is zero. These two observations establish the asserted identities.

A nondegenerate symmetric bilinear form determines
$\succ$ and $\prec$ through the invariance identities.
Their difference from the total product determines $\vee$.

\begin{thm}
\mlabel{thm:metric-splitting}
Let $(A,\circ)$ be a Leibniz algebra with a nondegenerate symmetric
bilinear form $\omega$. There is a unique quadratic NS-Leibniz
algebra on $A$ with form $\omega$ and total product $\circ$.
Its operations are determined by
\begin{equation}\mlabel{eq:metric-construct}
\begin{aligned}
 \omega(x\prec y,z)&=\omega(x,y\circ z+z\circ y),\\
 \omega(x\succ y,z)&=-\omega(y,x\circ z),\\
 x\vee y&=x\circ y-x\prec y-x\succ y.
\end{aligned}
\end{equation}
Every quadratic NS-Leibniz algebra arises in this way.
\end{thm}
\begin{proof}
Nondegeneracy of $\omega$ determines $\prec$ and $\succ$ by
\meqref{eq:metric-construct}, and $\vee$ is then determined by
$\circ=\prec+\succ+\vee$. Under $\omega^\sharp$, the pair
$(L_\succ,R_\prec)$ is identified with the coadjoint representation
$(L_\circ^*,-L_\circ^*-R_\circ^*)$. Its representation relations
imply \meqref{eq:ns3}, \meqref{eq:ns2} and \meqref{eq:ns1}.
Subtracting these identities from the Leibniz identity for $\circ$
establishes \meqref{eq:ns4}.

The first two invariance identities hold by construction;
\meqref{eq:inv-v-left} and \meqref{eq:inv-v-right} follow
from the preceding calculation.
Conversely, the defining invariance relations of a quadratic
NS-Leibniz algebra are \meqref{eq:metric-construct}, so uniqueness
proves the last statement.
\end{proof}

\begin{rmk}\mlabel{rem:phase}
The total product satisfies
\begin{equation}\mlabel{eq:defect}
 \omega(x\circ y,z)+\omega(y,x\circ z)-\omega(x,y\circ z+z\circ y)=\omega(x\vee y,z).
\end{equation}
Since $\omega$ is nondegenerate, the symplectic Leibniz identity
holds if and only if $\vee=0$. In that case, the quadratic
NS-Leibniz structure reduces to the symplectic
Leibniz-dendriform structure of \mcite{SunGuo}.
\end{rmk}

\begin{ex}\mlabel{ex:metric}
Let $A$ have basis $e,f$, with only nonzero total product $e\circ f=-f$.
Equip it with $\omega(e,e)=\omega(f,f)=1$, $\omega(e,f)=0$.
The operations in Theorem~\mref{thm:metric-splitting} have nonzero products
\[
 f\prec e=-f,\quad f\prec f=-e,\quad e\succ f=f,\qquad
 e\vee f=-2f,\quad f\vee e=f,\quad f\vee f=e.
\]
Here $\vee\ne0$, so $\omega$ is not symplectic for the total
Leibniz product.
\end{ex}

\begin{defi}\mlabel{def:representation}
A representation of an NS-Leibniz algebra $A$ consists of a vector space $V$
and six linear maps $l_\diamond,r_\diamond:A\to\mathrm{End}(V)$, $\diamond\in\{\succ,\prec,\vee\}$.
We denote it by
\[
(V,l_\succ,r_\succ,l_\prec,r_\prec,l_\vee,r_\vee)
\]
where $l_\circ=l_\succ+l_\prec+l_\vee$ and
$r_\circ=r_\succ+r_\prec+r_\vee$.
The four actions $l_\succ,r_\succ,l_\prec,r_\prec$ satisfy
the following relations for all $x,y\in A$:
\begin{align}
l_\prec(x)l_\circ (y)&=l_\prec(x\prec y)+l_\succ(y)l_\prec(x),\mlabel{eq:r1}\\
l_\prec(x)r_\circ (y)&=r_\prec(y)l_\prec(x)+r_\succ(x\prec y),\mlabel{eq:r2}\\
r_\prec(x\circ y)&=r_\prec(y)r_\prec(x)+l_\succ(x)r_\prec(y),\mlabel{eq:r3}\\
l_\succ(x)l_\prec(y)&=l_\prec(x\succ y)+l_\prec(y)l_\circ (x),\mlabel{eq:r4}\\
l_\succ(x)r_\prec(y)&=r_\prec(y)l_\succ(x)+r_\prec(x\circ y),\mlabel{eq:r5}\\
r_\succ(x\prec y)&=r_\prec(y)r_\succ(x)+l_\prec(x)r_\circ (y),\mlabel{eq:r6}\\
l_\succ(x)l_\succ(y)&=l_\succ(x\circ y)+l_\succ(y)l_\succ(x),\mlabel{eq:r7}\\
l_\succ(x)r_\succ(y)&=r_\succ(y)l_\circ (x)+r_\succ(x\succ y),\mlabel{eq:r8}\\
r_\succ(x\succ y)&=r_\succ(y)r_\circ (x)+l_\succ(x)r_\succ(y).\mlabel{eq:r9}
\end{align}
For the third operation, we require
\begin{align}
l_\succ(x)l_\vee(y)+l_\vee(x)l_\circ (y)
 &=l_\prec(x\vee y)+l_\vee(x\circ y)
   +l_\succ(y)l_\vee(x)+l_\vee(y)l_\circ (x),\mlabel{eq:r10}\\
l_\succ(x)r_\vee(y)+l_\vee(x)r_\circ (y)
 &=r_\prec(y)l_\vee(x)+r_\vee(y)l_\circ (x)
   +r_\succ(x\vee y)+r_\vee(x\circ y),\mlabel{eq:r11}\\
r_\succ(x\vee y)+r_\vee(x\circ y)
 &=r_\prec(y)r_\vee(x)+r_\vee(y)r_\circ (x)
   +l_\succ(x)r_\vee(y)+l_\vee(x)r_\circ (y).\mlabel{eq:r12}
\end{align}
\end{defi}

When $\vee=0$ and $l_\vee=r_\vee=0$, these relations
reduce to the Leibniz-dendriform representation identities
\mcite{SunGuo}. Relations
\meqref{eq:r10}--\meqref{eq:r12} correspond to
\meqref{eq:ns4} and can equivalently be replaced by the
requirement that $(V,l_\circ,r_\circ)$ be a representation of
$(A,\circ)$, as the next proposition shows.

\begin{pro}\mlabel{prop:semidirect}
Let $A$ be an NS-Leibniz algebra and let
$l_\diamond,r_\diamond:A\to\mathrm{End}(V)$ be linear maps.
They form a representation of $A$ if and only if $A\oplus V$ is an NS-Leibniz algebra under
\begin{equation}\mlabel{eq:semidirect}
 (x,u)\diamond(y,v)=(x\diamond y,l_\diamond(x)v+r_\diamond(y)u),
 \qquad \diamond\in\{\succ,\prec,\vee\}.
\end{equation}
Equivalently, it suffices to impose
\meqref{eq:r1}--\meqref{eq:r9} and require that
$(V,l_\circ,r_\circ)$ be a representation of $(A,\circ)$.
Moreover, $(V,l_\succ,r_\prec)$ is also a representation of
$(A,\circ)$.
\end{pro}
\begin{proof}
Substitute the products in \meqref{eq:semidirect} into
the four NS-Leibniz identities
\meqref{eq:ns1}--\meqref{eq:ns4}.
For each identity, place the $V$ entry successively
in the third, second and first positions, with the
other two entries in $A$.
These substitutions yield, respectively,
\meqref{eq:r1}--\meqref{eq:r3} for \meqref{eq:ns1},
\meqref{eq:r4}--\meqref{eq:r6} for \meqref{eq:ns2},
\meqref{eq:r7}--\meqref{eq:r9} for \meqref{eq:ns3},
and \meqref{eq:r10}--\meqref{eq:r12} for \meqref{eq:ns4}.
Triples with at least two entries in $V$ contribute zero,
whereas triples in $A$ satisfy the NS-Leibniz identities
by assumption. By trilinearity, this proves the first
statement.

Relations \meqref{eq:r1}--\meqref{eq:r9} express the first
three NS-Leibniz identities on $A\oplus V$. The total product is
\[
 (x,u)\circ(y,v)
 =(x\circ y,l_\circ(x)v+r_\circ(y)u).
\]
It is Leibniz if and only if $(V,l_\circ,r_\circ)$ is a
Leibniz representation. By the sum of the four identities in
Definition~\mref{def:ns}, its Leibniz identity and the first
three NS-Leibniz identities are equivalent to the fourth.
This proves the equivalent description.

For $(l_\succ,r_\prec)$, the first two Leibniz representation
relations are \meqref{eq:r7} and \meqref{eq:r5}. Adding \meqref{eq:r3} and \meqref{eq:r5} after moving all terms to the left yields
\[
 r_\prec(y)\bigl(r_\prec(x)+l_\succ(x)\bigr)=0.
\]
The pair $(l_\circ,r_\circ)$ is the representation associated with
the total product on $A\oplus V$.
\end{proof}

\begin{ex}
The regular representation of $A$ is
$(A,L_\succ,R_\succ,L_\prec,R_\prec,L_\vee,R_\vee)$.
The representation relations follow from
\meqref{eq:ns1}--\meqref{eq:ns4}.
By Proposition~\mref{prop:semidirect}, $(A,L_\succ,R_\prec)$ is
also a representation of $(A,\circ)$.
\end{ex}

The six dual actions are determined by the representation
identities and the natural pairing, as follows.

\begin{pro}\mlabel{thm:dual}
Let $(V,l_\succ,r_\succ,l_\prec,r_\prec,l_\vee,r_\vee)$ be a
representation of an NS-Leibniz algebra $A$. Then
\begin{equation}\mlabel{eq:dual-representation}
 \bigl(V^*,\,l_\circ^*,\,(l_\prec+r_\succ)^*,\,
       -l_\prec^*,\,-l_\star^*,\,-l_\vee^*,\,
       (l_\vee+r_\vee)^*\bigr)
\end{equation}
is also a representation of $A$.
The total left and right actions, and the left $\succ$ and
right $\prec$ actions, respectively, are
\begin{equation}\mlabel{eq:dual-total-partial}
 \bigl(l_\succ^*,-(l_\succ+r_\prec)^*\bigr),
 \qquad \bigl(l_\circ^*,-l_\star^*\bigr).
\end{equation}
Under the natural identification $V^{**}\simeq V$, the dual
of this representation is the original representation.
\end{pro}
\begin{proof}
The relations arising from
\meqref{eq:ns1}--\meqref{eq:ns3} are verified by the
transposition computations in \mcite{SunGuo},
Proposition~2.9, with
$l_\circ=l_\succ+l_\prec+l_\vee$ and
$r_\circ=r_\succ+r_\prec+r_\vee$.
It remains to verify the relations involving $\vee$.

For all $x,y\in A$, $\zeta\in V^*$ and $v\in V$, we have
\begin{align*}
&\Bigl\langle
\bigl(
-l_\circ^*(x)l_\vee^*(y)
-l_\vee^*(x)l_\succ^*(y)
+l_\prec^*(x\vee y)
+l_\vee^*(x\circ y)\\
&\hspace{35mm}
+l_\circ^*(y)l_\vee^*(x)
+l_\vee^*(y)l_\succ^*(x)
\bigr)\zeta,v
\Bigr\rangle\\
={}&\Bigl\langle\zeta,
\bigl(
l_\succ(x)l_\vee(y)
+l_\vee(x)l_\circ(y)
-l_\prec(x\vee y)
-l_\vee(x\circ y)\\
&\hspace{35mm}
-l_\succ(y)l_\vee(x)
-l_\vee(y)l_\circ(x)
\bigr)v
\Bigr\rangle
=0
\end{align*}
by \meqref{eq:r10}.
The other two relations involving $\vee$ are verified
by the same pairing calculation, using
\meqref{eq:r11} and \meqref{eq:r12}.
Hence \meqref{eq:dual-representation} is a representation.

Summing the three left actions and the three right
actions in \meqref{eq:dual-representation} yields
the first pair in \meqref{eq:dual-total-partial}.
The left $\succ$ and right $\prec$ actions yield
the second pair.
Applying \meqref{eq:dual-representation} once more,
one checks that the resulting six actions coincide
with the initial actions under the natural
identification $V^{**}\cong V$.
\end{proof}

\begin{cor}\mlabel{cor:dual-regular}
Let $A$ be an NS-Leibniz algebra. Then
\begin{equation}\mlabel{eq:dual-regular}
 \bigl(A^*,\,L_\circ^*,\,(L_\prec+R_\succ)^*,\,
       -L_\prec^*,\,-L_\star^*,\,-L_\vee^*,\,
       (L_\vee+R_\vee)^*\bigr)
\end{equation}
is a representation of $A$, called the dual regular
representation. When $\vee=0$, \meqref{eq:dual-regular} reduces to
the dual regular representation of the corresponding
Leibniz-dendriform algebra \mcite{SunGuo}.
\end{cor}

\begin{lem}
Let $A$ be an NS-Leibniz algebra with a nondegenerate
symmetric bilinear form $\omega$.
Then \meqref{eq:inv-p}--\meqref{eq:inv-v-right} hold
if and only if $\omega^\sharp:A\to A^*$ is an
isomorphism from the regular representation to
the dual regular representation
\meqref{eq:dual-regular}.
\end{lem}

\begin{proof}
The intertwining relations for $L_\succ$, $R_\prec$, $L_\vee$ and
$R_\vee$ are \meqref{eq:inv-s}, \meqref{eq:inv-p},
\meqref{eq:inv-v-left} and \meqref{eq:inv-v-right}, respectively.
The relation for $L_\prec$ follows from the symmetry in $y,z$ of
\meqref{eq:inv-p}. For $R_\succ$, it remains to establish
\[
\omega(x\succ y,z)=-\omega(x,z\succ y+y\prec z).
\]
By \meqref{eq:inv-s} and \meqref{eq:inv-p}, the right-hand side is
$\omega(y,z\circ x)-\omega(y,z\circ x+x\circ z)
=-\omega(y,x\circ z)$, which equals the left-hand side.
\end{proof}

\begin{defi}\mlabel{def:matched}
Let $A$ and $B$ be NS-Leibniz algebras. Assume that
$l_{\diamond A},r_{\diamond A}:A\to\mathrm{End}(B)$ and
$l_{\diamond B},r_{\diamond B}:B\to\mathrm{End}(A)$,
for $\diamond\in\{\succ,\prec,\vee\}$, define representations
of $A$ on $B$ and of $B$ on $A$, respectively.
The tuple
$(A,B,(l_{\diamond A},r_{\diamond A}),
(l_{\diamond B},r_{\diamond B}))_{\diamond\in\{\succ,\prec,\vee\}}$
is called a matched pair of NS-Leibniz algebras if the following
compatibility conditions hold for all $x,y\in A$ and $a,b\in B$:
\begin{align}
&x\succ(r_{\succ B}(a)y) +r_{\succ B}(l_{\succ A}(y)a)x
-r_{\succ B}(a)(x\circ y) -y\succ(r_{\succ B}(a)x)
-r_{\succ B}(l_{\succ A}(x)a)y=0\mlabel{eq:matched-cross}\\
&x\succ(l_{\succ B}(a)y) +r_{\succ B}(r_{\succ A}(y)a)x
-(r_{\circ B} (a)x)\succ y -l_{\succ B}(l_{\circ A} (x)a)y
-l_{\succ B}(a)(x\succ y)=0\\
&l_{\succ B}(a)(x\succ y) -(l_{\circ B} (a)x)\succ y
-l_{\succ B}(r_{\circ A} (x)a)y -x\succ(l_{\succ B}(a)y)
-r_{\succ B}(r_{\succ A}(y)a)x=0\\
&a\succ(r_{\succ A}(x)b) +r_{\succ A}(l_{\succ B}(b)x)a
-r_{\succ A}(x)(a\circ b) -b\succ(r_{\succ A}(x)a)
-r_{\succ A}(l_{\succ B}(a)x)b=0\\
&a\succ(l_{\succ A}(x)b) +r_{\succ A}(r_{\succ B}(b)x)a
-(r_{\circ A} (x)a)\succ b -l_{\succ A}(l_{\circ B} (a)x)b
-l_{\succ A}(x)(a\succ b)=0\\
&l_{\succ A}(x)(a\succ b) -(l_{\circ A} (x)a)\succ b
-l_{\succ A}(r_{\circ B} (a)x)b -a\succ(l_{\succ A}(x)b)
-r_{\succ A}(r_{\succ B}(b)x)a=0\\
&x\succ(r_{\prec B}(a)y) +r_{\succ B}(l_{\prec A}(y)a)x
-r_{\prec B}(a)(x\succ y) -y\prec(r_{\circ B} (a)x)
-r_{\prec B}(l_{\circ A} (x)a)y=0\\
&x\succ(l_{\prec B}(a)y) +r_{\succ B}(r_{\prec A}(y)a)x
-(r_{\succ B}(a)x)\prec y -l_{\prec B}(l_{\succ A}(x)a)y
-l_{\prec B}(a)(x\circ y)=0\\
&l_{\succ B}(a)(x\prec y) -(l_{\succ B}(a)x)\prec y
-l_{\prec B}(r_{\succ A}(x)a)y -x\prec(l_{\circ B} (a)y)
-r_{\prec B}(r_{\circ A} (y)a)x=0\\
&a\succ(r_{\prec A}(x)b) +r_{\succ A}(l_{\prec B}(b)x)a
-r_{\prec A}(x)(a\succ b) -b\prec(r_{\circ A} (x)a)
-r_{\prec A}(l_{\circ B} (a)x)b=0\\
&a\succ(l_{\prec A}(x)b) +r_{\succ A}(r_{\prec B}(b)x)a
-(r_{\succ A}(x)a)\prec b -l_{\prec A}(l_{\succ B}(a)x)b
-l_{\prec A}(x)(a\circ b)=0\\
&l_{\succ A}(x)(a\prec b) -(l_{\succ A}(x)a)\prec b
-l_{\prec A}(r_{\succ B}(a)x)b -a\prec(l_{\circ A} (x)b)
-r_{\prec A}(r_{\circ B} (b)x)a=0\\
&x\prec(r_{\circ B} (a)y) +r_{\prec B}(l_{\circ A} (y)a)x
-r_{\prec B}(a)(x\prec y) -y\succ(r_{\prec B}(a)x)
-r_{\succ B}(l_{\prec A}(x)a)y=0\\
&x\prec(l_{\circ B} (a)y) +r_{\prec B}(r_{\circ A} (y)a)x
-(r_{\prec B}(a)x)\prec y -l_{\prec B}(l_{\prec A}(x)a)y
-l_{\succ B}(a)(x\prec y)=0\\
&l_{\prec B}(a)(x\circ y) -(l_{\prec B}(a)x)\prec y
-l_{\prec B}(r_{\prec A}(x)a)y -x\succ(l_{\prec B}(a)y)
-r_{\succ B}(r_{\prec A}(y)a)x=0\\
&a\prec(r_{\circ A} (x)b) +r_{\prec A}(l_{\circ B} (b)x)a
-r_{\prec A}(x)(a\prec b) -b\succ(r_{\prec A}(x)a)
-r_{\succ A}(l_{\prec B}(a)x)b=0\\
&a\prec(l_{\circ A} (x)b) +r_{\prec A}(r_{\circ B} (b)x)a
-(r_{\prec A}(x)a)\prec b -l_{\prec A}(l_{\prec B}(a)x)b
-l_{\succ A}(x)(a\prec b)=0\\
&l_{\prec A}(x)(a\circ b) -(l_{\prec A}(x)a)\prec b
-l_{\prec A}(r_{\prec B}(a)x)b -a\succ(l_{\prec A}(x)b)
-r_{\succ A}(r_{\prec B}(b)x)a=0\\
&x\succ(r_{\vee B}(a)y)+r_{\succ B}(l_{\vee A}(y)a)x+x\vee(r_{\circ B} (a)y)+r_{\vee B}(l_{\circ A} (y)a)x-r_{\prec B}(a)(x\vee y)\notag\\*
 &-r_{\vee B}(a)(x\circ y)-y\succ(r_{\vee B}(a)x)-r_{\succ B}(l_{\vee A}(x)a)y-y\vee(r_{\circ B} (a)x)-r_{\vee B}(l_{\circ A} (x)a)y=0\\
&x\succ(l_{\vee B}(a)y)+r_{\succ B}(r_{\vee A}(y)a)x+x\vee(l_{\circ B} (a)y)+r_{\vee B}(r_{\circ A} (y)a)x-(r_{\vee B}(a)x)\prec y\notag\\*
 &-l_{\prec B}(l_{\vee A}(x)a)y-(r_{\circ B} (a)x)\vee y-l_{\vee B}(l_{\circ A} (x)a)y-l_{\succ B}(a)(x\vee y)-l_{\vee B}(a)(x\circ y)=0\\
&l_{\succ B}(a)(x\vee y)+l_{\vee B}(a)(x\circ y)-(l_{\vee B}(a)x)\prec y-l_{\prec B}(r_{\vee A}(x)a)y-(l_{\circ B} (a)x)\vee y\notag\\*
 &-l_{\vee B}(r_{\circ A} (x)a)y-x\succ(l_{\vee B}(a)y)-r_{\succ B}(r_{\vee A}(y)a)x-x\vee(l_{\circ B} (a)y)-r_{\vee B}(r_{\circ A} (y)a)x=0\\
&a\succ(r_{\vee A}(x)b)+r_{\succ A}(l_{\vee B}(b)x)a+a\vee(r_{\circ A} (x)b)+r_{\vee A}(l_{\circ B} (b)x)a-r_{\prec A}(x)(a\vee b)\notag\\*
 &-r_{\vee A}(x)(a\circ b)-b\succ(r_{\vee A}(x)a)-r_{\succ A}(l_{\vee B}(a)x)b-b\vee(r_{\circ A} (x)a)-r_{\vee A}(l_{\circ B} (a)x)b=0\\
&a\succ(l_{\vee A}(x)b)+r_{\succ A}(r_{\vee B}(b)x)a+a\vee(l_{\circ A} (x)b)+r_{\vee A}(r_{\circ B} (b)x)a-(r_{\vee A}(x)a)\prec b\notag\\*
 &-l_{\prec A}(l_{\vee B}(a)x)b-(r_{\circ A} (x)a)\vee b-l_{\vee A}(l_{\circ B} (a)x)b-l_{\succ A}(x)(a\vee b)-l_{\vee A}(x)(a\circ b)=0\\
&l_{\succ A}(x)(a\vee b)+l_{\vee A}(x)(a\circ b)-(l_{\vee A}(x)a)\prec b-l_{\prec A}(r_{\vee B}(a)x)b-(l_{\circ A} (x)a)\vee b\notag\\*
 &-l_{\vee A}(r_{\circ B} (a)x)b-a\succ(l_{\vee A}(x)b)-r_{\succ A}(r_{\vee B}(b)x)a-a\vee(l_{\circ A} (x)b)-r_{\vee A}(r_{\circ B} (b)x)a=0\mlabel{eq:matched-cross-last}
\end{align}

where $l_{\circ A}=l_{\succ A}+l_{\prec A}+l_{\vee A}$,
$r_{\circ A}=r_{\succ A}+r_{\prec A}+r_{\vee A}$,
$l_{\circ B}=l_{\succ B}+l_{\prec B}+l_{\vee B}$ and
$r_{\circ B}=r_{\succ B}+r_{\prec B}+r_{\vee B}$.
Define three binary operations on $A\oplus B$ by
\begin{equation}\mlabel{eq:bicrossed}
\begin{aligned}
 (x+a)\succ_D(y+b)
 &=x\succ_A y+l_{\succ B}(a)y+r_{\succ B}(b)x
   +a\succ_B b+l_{\succ A}(x)b+r_{\succ A}(y)a,\\
 (x+a)\prec_D(y+b)
 &=x\prec_A y+l_{\prec B}(a)y+r_{\prec B}(b)x
   +a\prec_B b+l_{\prec A}(x)b+r_{\prec A}(y)a,\\
 (x+a)\vee_D(y+b)
 &=x\vee_A y+l_{\vee B}(a)y+r_{\vee B}(b)x
   +a\vee_B b+l_{\vee A}(x)b+r_{\vee A}(y)a.
\end{aligned}
\end{equation}

The direct sum $A\oplus B$ equipped with these products is denoted by
$A\bowtie B$.
\end{defi}
The relations in each direction correspond to the four
NS-Leibniz identities. When both third products and their
associated actions vanish, they reduce to the matched-pair
conditions for Leibniz-dendriform algebras \mcite{SunGuo}.

\begin{pro}\mlabel{thm:matched}
Let $A,B$ be NS-Leibniz algebras with the six action maps
in each direction as above. The products \meqref{eq:bicrossed}
define an NS-Leibniz algebra on $A\oplus B$ if and only if
these data form a matched pair. Conversely, every NS-Leibniz
algebra that is the vector-space direct sum of two NS-Leibniz subalgebras
arises in this way.
If both action systems are representations, the last three
compatibility relations in each direction can equivalently
be replaced by the requirement that the total action pairs
form a matched pair of the subadjacent Leibniz algebras.
\end{pro}

\begin{proof}
Substitute the products \meqref{eq:bicrossed} into
the four defining identities of an NS-Leibniz algebra.
For triples with two entries in $A$ and one in $B$,
the $B$ components are precisely the representation
relations for the action of $A$ on $B$, while the
$A$ components are the twelve compatibility conditions
involving $x,y\in A$ and $a\in B$.
For triples with two entries in $B$ and one in $A$,
the $A$ components are the representation relations
for the action of $B$ on $A$, while the $B$ components
are the twelve compatibility conditions involving
$a,b\in B$ and $x\in A$.
Triples lying entirely in $A$ or entirely in $B$
satisfy the identities by assumption.
By trilinearity, this proves the first assertion.

Assume that both action systems are representations.
The nine compatibility conditions in each direction
corresponding to the first three NS-Leibniz identities
ensure that these identities hold on $A\oplus B$.
Adding the four NS-Leibniz identities yields
the Leibniz identity for the total product.
Hence, under the first three identities, the fourth
is equivalent to the Leibniz identity for the total
product, which is precisely the matched-pair condition
for the total action pairs.

Conversely, suppose that $D=A\oplus B$ is an
NS-Leibniz algebra and that $A$ and $B$ are
NS-Leibniz subalgebras.
Projecting the mixed products onto $A$ and $B$
defines the actions in \meqref{eq:bicrossed}.
The same expansions establish the representation
relations and all the compatibility conditions.
\end{proof}

\begin{cor}\mlabel{cor:underlying-matched}
A matched pair of NS-Leibniz algebras induces a matched pair of their
subadjacent Leibniz algebras with action pairs $(l_{\circ A} ,r_{\circ A} )$ and $(l_{\circ B} ,r_{\circ B} )$.
\end{cor}
\begin{proof}
The sum of the products in \meqref{eq:bicrossed} is the Leibniz
bicrossed product for $(l_{\circ A},r_{\circ A})$ and
$(l_{\circ B},r_{\circ B})$.
\end{proof}

\begin{defi}\mlabel{def:manin}
A Manin triple of NS-Leibniz algebras is a quadratic NS-Leibniz algebra
$(D,\succ,\prec,\vee,\omega)$ with complementary isotropic NS-Leibniz
subalgebras $A,B_0$, so $D=A\oplus B_0$ as vector spaces.
The form identifies $B_0$ with $A^*$; in this identification it is
\begin{equation}\mlabel{eq:canonical-form}
 \omega(x+\zeta,y+\eta)=\zeta(y)+\eta(x).
\end{equation}
\end{defi}

The dual regular representations determine the mixed
products in the double. The resulting structure is
characterized by the following theorem.

\begin{thm}\mlabel{thm:double}
Let $A,A^*$ be NS-Leibniz algebras. The following are equivalent:
\begin{enumerate}[label=(\roman*)]
\item They form a matched pair with the two dual regular action systems
\meqref{eq:dual-regular}.
\item Their subadjacent Leibniz algebras form a matched pair with
\begin{equation}\mlabel{eq:partial-pair}
 (l_A,r_A)=(L_{\succ,A}^*,-(L_{\succ,A}+R_{\prec,A})^*),\quad
 (l_{A^*},r_{A^*})=(L_{\succ,A^*}^*,-(L_{\succ,A^*}+R_{\prec,A^*})^*).
\end{equation}
\item The products
\begin{align}
(x+\zeta)\prec_D(y+\eta)
={}&x\prec y-L_{\prec,A^*}^*(\zeta)y
 -(L_{\circ ,A^*}+R_{\circ ,A^*})^*(\eta)x\notag\\
&+\zeta\prec \eta-L_{\prec,A}^*(x)\eta
 -(L_{\circ ,A}+R_{\circ ,A})^*(y)\zeta,\mlabel{eq:double-p}\\
(x+\zeta)\succ_D(y+\eta)
={}&x\succ y+L_{\circ ,A^*}^*(\zeta)y
 +(L_{\prec,A^*}+R_{\succ,A^*})^*(\eta)x\notag\\
&+\zeta\succ \eta+L_{\circ ,A}^*(x)\eta
 +(L_{\prec,A}+R_{\succ,A})^*(y)\zeta,\mlabel{eq:double-s}\\
(x+\zeta)\vee_D(y+\eta)
={}&x\vee y-L_{\vee,A^*}^*(\zeta)y
 +(L_{\vee,A^*}+R_{\vee,A^*})^*(\eta)x\notag\\
&+\zeta\vee \eta-L_{\vee,A}^*(x)\eta
 +(L_{\vee,A}+R_{\vee,A})^*(y)\zeta\mlabel{eq:double-v}
\end{align}
make $(D=A\oplus A^*,A,A^*,\omega)$ a Manin triple, with
$\omega$ given by \meqref{eq:canonical-form}.
\end{enumerate}
\end{thm}
\begin{proof}
The implication (i) to (ii) follows from
\meqref{eq:dual-total-partial} and Corollary~\mref{cor:underlying-matched}.
Suppose (ii) holds. Equip the Leibniz bicrossed product $D$ with
\meqref{eq:canonical-form} and apply
Theorem~\mref{thm:metric-splitting}. For $x,y,z\in A$,
$\omega(x\prec_Dy,z)=\omega(x,y\circ z+z\circ y)=0$ and
$\omega(x\succ_Dy,z)=-\omega(y,x\circ z)=0$.
Since $A^\perp=A$, both products belong to $A$; their residual
$\vee_D$ also belongs to $A$. The same argument applies to $A^*$.

The restricted products agree with the given products. Indeed,
for $\zeta\in A^*$,
\[
\omega(x\succ_Dy,\zeta)=-\omega(y,x\circ_D \zeta)
 =-\langle L_{\succ,A}^*(x)\zeta,y\rangle=\langle \zeta,x\succ_Ay\rangle.
\]
The $A^*$ component of $y\circ_D\zeta+\zeta\circ_Dy$ is
$-R_{\prec,A}^*(y)\zeta$, so
$\omega(x\prec_Dy,\zeta)=\langle \zeta,x\prec_Ay\rangle$.
Subtracting from the total product yields the equality for $\vee$.
Pairing mixed products with each summand and applying
\meqref{eq:inv-p}--\meqref{eq:inv-v-right} yields
\meqref{eq:double-p}--\meqref{eq:double-v}. This proves (iii).
The implication (iii) to (i) follows from Proposition~\mref{thm:matched}.
\end{proof}

\section{NS-Leibniz bialgebras}
\mlabel{sec:manin-bialgebras}
\subsection{NS-Leibniz bialgebras}
We define the coalgebra structure by dualizing the four
NS-Leibniz identities. When $\Delta_\vee=0$, these
identities reduce to the Leibniz-dendriform coalgebra
identities \mcite{SunGuo}.

\begin{defi}\mlabel{def:coalgebra}
An NS-Leibniz coalgebra consists of three maps
$\Delta_\succ,\Delta_\prec,\Delta_\vee:A\to A\o A$ satisfying,
with $\Delta=\Delta_\prec+\Delta_\succ+\Delta_\vee$,
\begin{align}
(I\o\Delta)\Delta_\prec
 &=(\Delta_\prec\o I)\Delta_\prec
   +(\tau\o I)(I\o\Delta_\prec)\Delta_\succ,\mlabel{eq:c1}\\
(I\o\Delta_\prec)\Delta_\succ
 &=(\Delta_\succ\o I)\Delta_\prec
   +(\tau\o I)(I\o\Delta)\Delta_\prec,\mlabel{eq:c2}\\
(I\o\Delta_\succ)\Delta_\succ
 &=(\Delta\o I)\Delta_\succ
   +(\tau\o I)(I\o\Delta_\succ)\Delta_\succ,\mlabel{eq:c3}\\
(I\o\Delta_\vee)\Delta_\succ+(I\o\Delta)\Delta_\vee
 &=(\Delta_\vee\o I)\Delta_\prec+(\Delta\o I)\Delta_\vee\notag\\
 &\quad+(\tau\o I)(I\o\Delta_\vee)\Delta_\succ
       +(\tau\o I)(I\o\Delta)\Delta_\vee.\mlabel{eq:c4}
\end{align}
\end{defi}
\begin{lem}\mlabel{lem:coalgebra-dual}
The coalgebra identities above are equivalent to an NS-Leibniz structure
on $A^*$ defined by
\begin{equation}\mlabel{eq:coalgebra-pair}
 \langle \Delta_\diamond(x),\zeta\o \eta\rangle=\langle \zeta\diamond \eta,x\rangle,
 \qquad \diamond\in\{\succ,\prec,\vee\}.
\end{equation}
Equivalently, it suffices to impose
\meqref{eq:c1}--\meqref{eq:c3} and require the total
coproduct $\Delta$ to satisfy the Leibniz coalgebra identity
\begin{equation}\mlabel{eq:total-leibniz-coalgebra}
 (I\o\Delta)\Delta
 =(\Delta\o I)\Delta+(\tau\o I)(I\o\Delta)\Delta.
\end{equation}
\end{lem}
\begin{proof}
Pair \meqref{eq:c1}--\meqref{eq:c4} with
$\zeta\o\eta\o h$. For the first identity, this yields
\[
 \zeta\prec(\eta\circ h)
 =(\zeta\prec\eta)\prec h+\eta\succ(\zeta\prec h).
\]
The other three identities yield
\meqref{eq:ns2}--\meqref{eq:ns4} on $A^*$. Nondegeneracy of the
pairing proves the equivalence.
Summing \meqref{eq:c1}--\meqref{eq:c4} yields
\meqref{eq:total-leibniz-coalgebra}; conversely, subtracting
the first three identities from it yields the fourth.
\end{proof}

For the following tensor identities, set
\begin{equation}\mlabel{eq:odot}
x\odot y=x\succ y+y\prec x,\quad
L_\odot=L_\succ+R_\prec,\quad
R_\odot=R_\succ+L_\prec,\quad
\Delta_\odot=\Delta_\succ+\tau\Delta_\prec.
\end{equation}

An NS-Leibniz bialgebra consists of compatible algebra and
coalgebra structures. The compatibility conditions below
involve the total product $\circ=\succ+\prec+\vee$ and
the total coproduct
$\Delta=\Delta_\succ+\Delta_\prec+\Delta_\vee$.

\begin{defi}\mlabel{def:bialgebra}
An NS-Leibniz bialgebra is an NS-Leibniz algebra together with an
NS-Leibniz coalgebra on the same space such that, for all $x,y\in A$,
\begin{align}
0={}&\Delta(x\odot y)-(I\o L_\odot(x))\Delta(y)
 +\tau(I\o L_\odot(x))\Delta(y)\notag\\
 &+\tau(I\o R_\odot(y))\Delta_\succ(x)
 -(I\o R_\odot(y))\Delta_\succ(x),\mlabel{eq:k1}\\
0={}&\Delta(x\succ y)
 -(L_\succ(x)\o I+I\o L_\succ(x))\Delta(y)\notag\\
 &-(I\o R_\succ(y))\Delta_\odot(x)+\tau(I\o R_\odot(y))\Delta_\odot(x),\mlabel{eq:k2}\\
0={}&(I\o R_\succ(y))\tau\Delta_\prec(x)
 -(R_\prec(x)\o I)\Delta(y),\mlabel{eq:k3}\\
0={}&\Delta_\odot(x\circ y)-(I\o L_\circ (x)+L_\succ(x)\o I)\Delta_\odot(y)\notag\\
 &+(I\o L_\circ (y)+L_\succ(y)\o I)\Delta_\odot(x),\mlabel{eq:k4}\\
0={}&\Delta_\succ(x\circ y)-(I\o R_\circ (y))\Delta_\succ(x)\notag\\
 &-(I\o L_\circ (x)+L_\odot(x)\o I)\Delta_\succ(y)
 +(L_\odot(y)\o I)\Delta_\odot(x),\mlabel{eq:k5}\\
0={}&(I\o R_\circ (y))\tau\Delta_\prec(x)
 -(R_\prec(x)\o I)\Delta_\succ(y).\mlabel{eq:k6}
\end{align}
\end{defi}

\begin{thm}\mlabel{thm:bialgebra}
Let $A$ be an NS-Leibniz algebra and let
$(\Delta_\succ,\Delta_\prec,\Delta_\vee)$ be an
NS-Leibniz coalgebra. Then
$(A,\succ,\prec,\vee,\Delta_\succ,\Delta_\prec,\Delta_\vee)$
is an NS-Leibniz bialgebra if and only if
$(A,\circ)$ and $(A^*,\circ)$ form a matched pair of
Leibniz algebras with the actions in \meqref{eq:partial-pair}.
\end{thm}

\begin{proof}
By Lemma~\mref{lem:coalgebra-dual}, the coproducts
$(\Delta_\succ,\Delta_\prec,\Delta_\vee)$ induce an
NS-Leibniz algebra structure on $A^*$.
The action pairs in \meqref{eq:partial-pair} are
representations of the corresponding subadjacent
Leibniz algebras.

The pairing computations corresponding to
\meqref{eq:k1}--\meqref{eq:k3} are the same as those
in the proof of \mcite{SunGuo}, Theorem~3.4,
with $\Delta=\Delta_\succ+\Delta_\prec+\Delta_\vee$.
We therefore only verify the conditions involving
the total multiplication on $A$.

For all $x,y\in A$ and $\zeta,\eta\in A^*$, we have
\[
\begin{aligned}
&\Big\langle\eta,\,
r_{A^*}(\zeta)(x\circ y)
-x\circ(r_{A^*}(\zeta)y)
+y\circ(r_{A^*}(\zeta)x)\\
&\qquad
-r_{A^*}(l_A(y)\zeta)x
+r_{A^*}(l_A(x)\zeta)y
\Big\rangle\\
={}&\Big\langle
\Delta_\odot(x\circ y)
-(I\otimes L_\circ(x)+L_\succ(x)\otimes I)
 \Delta_\odot(y)\\
&\qquad
+(I\otimes L_\circ(y)+L_\succ(y)\otimes I)
 \Delta_\odot(x),
\zeta\otimes\eta
\Big\rangle.
\end{aligned}
\]
Thus this matched-pair condition is equivalent to
\meqref{eq:k4}. Similarly,
\[
\begin{aligned}
&\Big\langle\eta,\,
-l_{A^*}(\zeta)(x\circ y)
+(l_{A^*}(\zeta)x)\circ y
+x\circ(l_{A^*}(\zeta)y)\\
&\qquad
+l_{A^*}(r_A(x)\zeta)y
+r_{A^*}(r_A(y)\zeta)x
\Big\rangle\\
={}&\Big\langle
\Delta_\succ(x\circ y)
-(I\otimes R_\circ(y))\Delta_\succ(x)\\
&\qquad
-(I\otimes L_\circ(x)+L_\odot(x)\otimes I)
 \Delta_\succ(y)
+(L_\odot(y)\otimes I)\Delta_\odot(x),
\zeta\otimes\eta
\Big\rangle.
\end{aligned}
\]
Hence this condition is equivalent to
\meqref{eq:k5}.

Adding the mixed Leibniz identities with arguments
$(x,\zeta,y)$ and $(\zeta,x,y)$ and taking their
$A$ components, we obtain
\[
\bigl((l_{A^*}(\zeta)+r_{A^*}(\zeta))x\bigr)\circ y
+l_{A^*}\bigl((l_A(x)+r_A(x))\zeta\bigr)y=0.
\]
Moreover,
\[
\begin{aligned}
&\Big\langle\eta,\,
\bigl((l_{A^*}(\zeta)+r_{A^*}(\zeta))x\bigr)\circ y
+l_{A^*}\bigl((l_A(x)+r_A(x))\zeta\bigr)y
\Big\rangle\\
={}&\Big\langle
(I\otimes R_\circ(y))\tau\Delta_\prec(x)
-(R_\prec(x)\otimes I)\Delta_\succ(y),
\zeta\otimes\eta
\Big\rangle.
\end{aligned}
\]
This condition is therefore equivalent to
\meqref{eq:k6}.

The remaining components hold by the representation
identities. In each direction, replacing one mixed
identity by its sum with another is an equivalent
transformation. By nondegeneracy of the natural
pairings, all the matched-pair conditions are
equivalent to \meqref{eq:k1}--\meqref{eq:k6}.
This proves the assertion.
\end{proof}

\begin{cor}\mlabel{cor:bialgebra-characterizations}
Let $A$ be an NS-Leibniz algebra and let
$(\Delta_\succ,\Delta_\prec,\Delta_\vee)$ be an NS-Leibniz coalgebra.
Then the following conditions are equivalent:
\begin{enumerate}[label=(\roman*)]
\item $(A,\succ,\prec,\vee,\Delta_\succ,\Delta_\prec,\Delta_\vee)$
is an NS-Leibniz bialgebra.
\item $(A,\circ )$ and $(A^*,\circ )$ form a Leibniz matched pair with
\meqref{eq:partial-pair}.
\item $A,A^*$ form an NS-Leibniz matched pair with their dual regular actions.
\item \meqref{eq:double-p}--\meqref{eq:double-v} and
\meqref{eq:canonical-form} define a Manin triple.
\end{enumerate}
The algebra in (iv) is called the canonical double of the bialgebra.
\end{cor}
\begin{proof}
Combine Theorems~\mref{thm:double} and~\mref{thm:bialgebra}.
\end{proof}

\begin{rmk}
The third product and coproduct enter the six compatibility
conditions through $\circ$ and $\Delta$.
No further independent bialgebra compatibility condition is
needed for $\vee$. Its mixed identities follow from the
Leibniz double and the quadratic splitting in
Theorem~\mref{thm:double}.
When $\vee=0$ and $\Delta_\vee=0$, the definition reduces
to that of a Leibniz-dendriform bialgebra \mcite{SunGuo},
and the canonical double is a symplectic Leibniz phase space.
\end{rmk}

\begin{ex}
For any NS-Leibniz algebra $A$, the semidirect product
$A\ltimes A^*$ for the dual regular representation, with zero
products on $A^*$, is quadratic for the canonical form.
The associated Manin triple determines the bialgebra
$(A,\succ,\prec,\vee,0,0,0)$.
\end{ex}

\subsection{Coboundary NS-Leibniz bialgebras and the Yang--Baxter equation}
\mlabel{sec:coboundary}
\begin{defi}\mlabel{def:coboundary}
Let $(A,\succ,\prec,\vee)$ be an NS-Leibniz algebra.
For $r\in A\o A$, put
\begin{align}
\Delta_{\succ,r}(x)
 &=(L_\odot(x)\o I-I\o R_\circ (x))r,\mlabel{eq:delta-r-s}\\
\Delta_{\prec,r}(x)
 &=(L_\star(x)\o I-I\o R_\prec(x))\tau(r),\mlabel{eq:delta-r-p}\\
\Delta_{\vee,r}(x)
 &=((L_\vee(x)+R_\vee(x))\o I+I\o R_\vee(x))r.\mlabel{eq:delta-r-v}
\end{align}
An NS-Leibniz bialgebra is coboundary if its three coproducts have this
form for a common tensor $r$. Thus the maps in \meqref{eq:delta-r-s}--\meqref{eq:delta-r-v}
define a coboundary bialgebra when they satisfy the
coalgebra and compatibility relations of
Definition~\mref{def:bialgebra}.
\end{defi}

Write $\Delta_r=\Delta_{\succ,r}+\Delta_{\prec,r}+\Delta_{\vee,r}$,
$\Delta_{\odot,r}=\Delta_{\succ,r}+\tau\Delta_{\prec,r}$ and
$q=r+\tau(r)$. We also write $x\star y=x\circ y+y\circ x$,
so that $L_\star(x)y=x\star y$.
On $A^{\o3}$, put $\tau_{12}=\tau\o I$,
$\tau_{23}=I\o\tau$ and $\rho=\tau_{12}\tau_{23}$.
Thus $\rho(x\o y\o z)=z\o x\o y$.

Substituting the three coboundary maps into the coalgebra
identities yields the following criterion. We display
the equations for all four NS-Leibniz identities.

Write $\Delta_r=\Delta_{\succ,r}+\Delta_{\prec,r}+\Delta_{\vee,r}$ and $\Delta_{\odot,r}=\Delta_{\succ,r}+\tau\Delta_{\prec,r}$. We use $x\star y=x\circ y+y\circ x$, so that $L_\star(x)=L_\circ(x)+R_\circ(x)$.

\begin{pro}\mlabel{prop:general-coboundary-coalgebra}
Let $r=\sum_i a_i\otimes b_i\in A\otimes A$, and define the three maps by \meqref{eq:delta-r-s}--\meqref{eq:delta-r-v}. Then \meqref{eq:c1}, \meqref{eq:c2}, \meqref{eq:c3} and \meqref{eq:c4} hold if and only if, respectively, the following equations hold for every $x\in A$:
\begin{align}
0={}&\sum_{i,j}\Bigl[
 -(x\star b_i)\otimes(a_i\prec a_j)\otimes b_j
 -(x\star b_i)\otimes(a_j\succ a_i)\otimes b_j
 \notag\\
&\quad
 -(x\star b_i)\otimes a_j\otimes(b_j\succ a_i)
 +b_i\otimes((a_i\prec x)\prec a_j)\otimes b_j
 \notag\\
&\quad
 +b_i\otimes(a_j\succ(a_i\prec x))\otimes b_j
 +b_i\otimes a_j\otimes(b_j\succ(a_i\prec x))
 \notag\\
&\quad
 -((x\star b_i)\star b_j)\otimes a_j\otimes a_i
 +b_j\otimes(a_j\prec(x\star b_i))\otimes a_i
 \notag\\
&\quad
 +(b_i\star b_j)\otimes a_j\otimes(a_i\prec x)
 -b_j\otimes(a_j\prec b_i)\otimes(a_i\prec x)
 \notag\\
&\quad
 -(b_i\star b_j)\otimes(x\odot a_i)\otimes a_j
 +b_j\otimes(x\odot a_i)\otimes(a_j\prec b_i)
 \notag\\
&\quad
 +((b_i\circ x)\star b_j)\otimes a_i\otimes a_j
 -b_j\otimes a_i\otimes(a_j\prec(b_i\circ x))
 \Bigr]
 \notag\\
&+\sum_i\Bigl[
 (x\star b_i)\otimes
 \bigl((L_\star(a_i)\otimes I-I\otimes R_\prec(a_i))
 (r+\tau(r))\bigr)
 \notag\\
&\quad
 -b_i\otimes
 \bigl((L_\star(a_i\prec x)\otimes I
 -I\otimes R_\prec(a_i\prec x))(r+\tau(r))\bigr)
 \Bigr],
 \mlabel{eq:general-cb-coalgebra-p}\\
0={}&\sum_{i,j}\Bigl[
 (x\odot a_i)\otimes(b_i\star b_j)\otimes a_j
 -(x\odot a_i)\otimes b_j\otimes(a_j\prec b_i)
 \notag\\
&\quad
 -a_i\otimes((b_i\circ x)\star b_j)\otimes a_j
 +a_i\otimes b_j\otimes(a_j\prec(b_i\circ x))
 \notag\\
&\quad
 -((x\star b_i)\odot a_j)\otimes b_j\otimes a_i
 +a_j\otimes(b_j\circ(x\star b_i))\otimes a_i
 \notag\\
&\quad
 +(b_i\odot a_j)\otimes b_j\otimes(a_i\prec x)
 -a_j\otimes(b_j\circ b_i)\otimes(a_i\prec x)
 \notag\\
&\quad
 +(a_i\prec a_j)\otimes(x\star b_i)\otimes b_j
 +(a_j\succ a_i)\otimes(x\star b_i)\otimes b_j
 \notag\\
&\quad
 +a_j\otimes(x\star b_i)\otimes(b_j\succ a_i)
 -((a_i\prec x)\prec a_j)\otimes b_i\otimes b_j
 \notag\\
&\quad
 -(a_j\succ(a_i\prec x))\otimes b_i\otimes b_j
 -a_j\otimes b_i\otimes(b_j\succ(a_i\prec x))
 \Bigr]
 \notag\\
&-(\tau\otimes I)\sum_i\Bigl[
 (x\star b_i)\otimes
 \bigl((L_\star(a_i)\otimes I-I\otimes R_\prec(a_i))
 (r+\tau(r))\bigr)
 \notag\\
&\quad
 -b_i\otimes
 \bigl((L_\star(a_i\prec x)\otimes I
 -I\otimes R_\prec(a_i\prec x))(r+\tau(r))\bigr)
 \Bigr],
 \mlabel{eq:general-cb-coalgebra-sp}\\
0={}&\sum_{i,j}\Bigl[
 -((x\odot a_i)\prec a_j)\otimes b_j\otimes b_i
 -(a_j\succ(x\odot a_i))\otimes b_j\otimes b_i
 \notag\\
&\quad
 -a_j\otimes(b_j\succ(x\odot a_i))\otimes b_i
 +(a_i\prec a_j)\otimes b_j\otimes(b_i\circ x)
 \notag\\
&\quad
 +(a_j\succ a_i)\otimes b_j\otimes(b_i\circ x)
 +a_j\otimes(b_j\succ a_i)\otimes(b_i\circ x)
 \notag\\
&\quad
 -(x\odot a_i)\otimes(b_i\odot a_j)\otimes b_j
 +(x\odot a_i)\otimes a_j\otimes(b_j\circ b_i)
 \notag\\
&\quad
 +a_i\otimes((b_i\circ x)\odot a_j)\otimes b_j
 -a_i\otimes a_j\otimes(b_j\circ(b_i\circ x))
 \notag\\
&\quad
 +(b_i\odot a_j)\otimes(x\odot a_i)\otimes b_j
 -a_j\otimes(x\odot a_i)\otimes(b_j\circ b_i)
 \notag\\
&\quad
 -((b_i\circ x)\odot a_j)\otimes a_i\otimes b_j
 +a_j\otimes a_i\otimes(b_j\circ(b_i\circ x))
 \Bigr]
 \notag\\
&+\sum_i\Bigl[
 \bigl((L_\star(x\odot a_i)\otimes I
 -I\otimes R_\prec(x\odot a_i))(r+\tau(r))\bigr)
 \otimes b_i
 \notag\\
&\quad
 -\bigl((L_\star(a_i)\otimes I-I\otimes R_\prec(a_i))
 (r+\tau(r))\bigr)\otimes(b_i\circ x)
 \Bigr],
 \mlabel{eq:general-cb-coalgebra-s}\\
0={}&(I^{\otimes3}-\tau\otimes I)
 \sum_{i,j}\Bigl[
 (x\odot a_i)\otimes(b_i\vee a_j)\otimes b_j
 +(x\odot a_i)\otimes(a_j\vee b_i)\otimes b_j
 \notag\\
&\quad
 +(x\odot a_i)\otimes a_j\otimes(b_j\vee b_i)
 -a_i\otimes((b_i\circ x)\vee a_j)\otimes b_j
 \notag\\
&\quad
 -a_i\otimes(a_j\vee(b_i\circ x))\otimes b_j
 -a_i\otimes a_j\otimes(b_j\vee(b_i\circ x))
 \notag\\
&\quad
 -(x\vee a_i)\otimes(b_i\prec a_j)\otimes b_j
 -(x\vee a_i)\otimes(a_j\succ b_i)\otimes b_j
 \notag\\
&\quad
 -(x\vee a_i)\otimes a_j\otimes(b_j\succ b_i)
 -(a_i\vee x)\otimes(b_i\prec a_j)\otimes b_j
 \notag\\
&\quad
 -(a_i\vee x)\otimes(a_j\succ b_i)\otimes b_j
 -(a_i\vee x)\otimes a_j\otimes(b_j\succ b_i)
 \notag\\
&\quad
 -a_i\otimes((b_i\vee x)\prec a_j)\otimes b_j
 -a_i\otimes(a_j\succ(b_i\vee x))\otimes b_j
 \notag\\
&\quad
 -a_i\otimes a_j\otimes(b_j\succ(b_i\vee x))
 \Bigr]
 \notag\\
&+(I^{\otimes3}-\tau\otimes I)\sum_i\Bigl[
 (x\vee a_i+a_i\vee x)\otimes
 \bigl((L_\star(b_i)\otimes I-I\otimes R_\prec(b_i))
 (r+\tau(r))\bigr)
 \notag\\
&\quad
 +a_i\otimes
 \bigl((L_\star(b_i\vee x)\otimes I
 -I\otimes R_\prec(b_i\vee x))(r+\tau(r))\bigr)
 \Bigr]
 \notag\\
&-\sum_{i,j}\Bigl[
 ((x\star b_i)\vee a_j)\otimes b_j\otimes a_i
 +(a_j\vee(x\star b_i))\otimes b_j\otimes a_i
 \notag\\
&\quad
 +a_j\otimes(b_j\vee(x\star b_i))\otimes a_i
 -(b_i\vee a_j)\otimes b_j\otimes(a_i\prec x)
 \notag\\
&\quad
 -(a_j\vee b_i)\otimes b_j\otimes(a_i\prec x)
 -a_j\otimes(b_j\vee b_i)\otimes(a_i\prec x)
 \notag\\
&\quad
 -((x\vee a_i)\prec a_j)\otimes b_j\otimes b_i
 -(a_j\succ(x\vee a_i))\otimes b_j\otimes b_i
 \notag\\
&\quad
 -a_j\otimes(b_j\succ(x\vee a_i))\otimes b_i
 -((a_i\vee x)\prec a_j)\otimes b_j\otimes b_i
 \notag\\
&\quad
 -(a_j\succ(a_i\vee x))\otimes b_j\otimes b_i
 -a_j\otimes(b_j\succ(a_i\vee x))\otimes b_i
 \notag\\
&\quad
 -(a_i\prec a_j)\otimes b_j\otimes(b_i\vee x)
 -(a_j\succ a_i)\otimes b_j\otimes(b_i\vee x)
 \notag\\
&\quad
 -a_j\otimes(b_j\succ a_i)\otimes(b_i\vee x)
 \Bigr]
 \notag\\
&-\sum_i\Bigl[
 \bigl((L_\star(x\vee a_i+a_i\vee x)\otimes I
 -I\otimes R_\prec(x\vee a_i+a_i\vee x))
 (r+\tau(r))\bigr)\otimes b_i
 \notag\\
&\quad
 +\bigl((L_\star(a_i)\otimes I-I\otimes R_\prec(a_i))
 (r+\tau(r))\bigr)\otimes(b_i\vee x)
 \Bigr].
 \mlabel{eq:general-cb-coalgebra-v}
\end{align}
\end{pro}

\begin{proof}
The coboundary formulas yield
\begin{align*}
\Delta_{\succ,r}(x)
 &=\sum_i\bigl[
 (x\odot a_i)\otimes b_i-a_i\otimes(b_i\circ x)
 \bigr],\\
\Delta_{\prec,r}(x)
 &=\sum_i\bigl[
 (x\star b_i)\otimes a_i-b_i\otimes(a_i\prec x)
 \bigr],\\
\Delta_{\vee,r}(x)
 &=\sum_i\bigl[
 (x\vee a_i+a_i\vee x)\otimes b_i
 +a_i\otimes(b_i\vee x)
 \bigr].
\end{align*}
Adding these equalities, we obtain
\begin{equation}\mlabel{eq:general-total-coboundary}
\begin{aligned}
\Delta_r(x)
={}&-\sum_i\bigl[
 (x\prec a_i+a_i\succ x)\otimes b_i
 +a_i\otimes(b_i\succ x)
 \bigr]\\
&+(L_\star(x)\otimes I-I\otimes R_\prec(x))(r+\tau(r)).
\end{aligned}
\end{equation}
For \meqref{eq:c3}, direct substitution yields
\[
\begin{aligned}
&(\Delta_r\otimes I)\Delta_{\succ,r}(x)
 -(I\otimes\Delta_{\succ,r})\Delta_{\succ,r}(x)
 +(\tau\otimes I)(I\otimes\Delta_{\succ,r})
 \Delta_{\succ,r}(x)\\
&=I(1)+I(2)+I(3),
\end{aligned}
\]
where
\begin{align*}
I(1)={}&\sum_{i,j}\Bigl[
 -((x\odot a_i)\prec a_j)\otimes b_j\otimes b_i
 -(a_j\succ(x\odot a_i))\otimes b_j\otimes b_i\\
&\quad
 -(x\odot a_i)\otimes(b_i\odot a_j)\otimes b_j
 +(x\odot a_i)\otimes a_j\otimes(b_j\circ b_i)\\
&\quad
 -((b_i\circ x)\odot a_j)\otimes a_i\otimes b_j
 \Bigr]\\
&+\sum_i
 \bigl((L_\star(x\odot a_i)\otimes I
 -I\otimes R_\prec(x\odot a_i))(r+\tau(r))\bigr)
 \otimes b_i,\\
I(2)={}&\sum_{i,j}\Bigl[
 -a_j\otimes(b_j\succ(x\odot a_i))\otimes b_i
 +a_i\otimes((b_i\circ x)\odot a_j)\otimes b_j\\
&\quad
 +(b_i\odot a_j)\otimes(x\odot a_i)\otimes b_j
 -a_j\otimes(x\odot a_i)\otimes(b_j\circ b_i)
 \Bigr],\\
I(3)={}&\sum_{i,j}\Bigl[
 (a_i\prec a_j)\otimes b_j\otimes(b_i\circ x)
 +(a_j\succ a_i)\otimes b_j\otimes(b_i\circ x)\\
&\quad
 +a_j\otimes(b_j\succ a_i)\otimes(b_i\circ x)
 -a_i\otimes a_j\otimes(b_j\circ(b_i\circ x))\\
&\quad
 +a_j\otimes a_i\otimes(b_j\circ(b_i\circ x))
 \Bigr]\\
&-\sum_i
 \bigl((L_\star(a_i)\otimes I-I\otimes R_\prec(a_i))
 (r+\tau(r))\bigr)\otimes(b_i\circ x).
\end{align*}
Adding these expressions proves the equivalence between \meqref{eq:c3} and \meqref{eq:general-cb-coalgebra-s}.

For \meqref{eq:c1} and \meqref{eq:c2}, the differences between the two sides are, respectively,
\begin{align*}
&\sum_i\Bigl[
 (x\star b_i)\otimes\Delta_r(a_i)
 -b_i\otimes\Delta_r(a_i\prec x)\\
&\qquad
 -\Delta_{\prec,r}(x\star b_i)\otimes a_i
 +\Delta_{\prec,r}(b_i)\otimes(a_i\prec x)
 \Bigr]\\
&\quad
 -(\tau\otimes I)\sum_i\Bigl[
 (x\odot a_i)\otimes\Delta_{\prec,r}(b_i)
 -a_i\otimes\Delta_{\prec,r}(b_i\circ x)
 \Bigr],
\end{align*}
and
\begin{align*}
&\sum_i\Bigl[
 (x\odot a_i)\otimes\Delta_{\prec,r}(b_i)
 -a_i\otimes\Delta_{\prec,r}(b_i\circ x)\\
&\qquad
 -\Delta_{\succ,r}(x\star b_i)\otimes a_i
 +\Delta_{\succ,r}(b_i)\otimes(a_i\prec x)
 \Bigr]\\
&\quad
 -(\tau\otimes I)\sum_i\Bigl[
 (x\star b_i)\otimes\Delta_r(a_i)
 -b_i\otimes\Delta_r(a_i\prec x)
 \Bigr].
\end{align*}
Expanding the coproducts yields \meqref{eq:general-cb-coalgebra-p} and \meqref{eq:general-cb-coalgebra-sp}.

For \meqref{eq:c4}, the difference between the two sides is
\begin{align*}
&(I^{\otimes3}-\tau\otimes I)\sum_i\Bigl[
 (x\odot a_i)\otimes\Delta_{\vee,r}(b_i)
 -a_i\otimes\Delta_{\vee,r}(b_i\circ x)\\
&\qquad
 +(x\vee a_i+a_i\vee x)\otimes\Delta_r(b_i)
 +a_i\otimes\Delta_r(b_i\vee x)
 \Bigr]\\
&-\sum_i\Bigl[
 \Delta_{\vee,r}(x\star b_i)\otimes a_i
 -\Delta_{\vee,r}(b_i)\otimes(a_i\prec x)\\
&\qquad
 +\Delta_r(x\vee a_i+a_i\vee x)\otimes b_i
 +\Delta_r(a_i)\otimes(b_i\vee x)
 \Bigr].
\end{align*}
Expanding $\Delta_{\vee,r}$ and $\Delta_r$ yields \meqref{eq:general-cb-coalgebra-v}. This completes the proof.
\end{proof}

\begin{pro}\mlabel{prop:general-coboundary-compatibility}
Let $r\in A\otimes A$, and define the three coboundary maps by \meqref{eq:delta-r-s}--\meqref{eq:delta-r-v}. Then \meqref{eq:k4}--\meqref{eq:k6} hold automatically. The remaining conditions \meqref{eq:k1}--\meqref{eq:k3} are equivalent, respectively, to
\begin{align}
0={}&\Bigl[
 \bigl(L_\star(x\odot y)
 -L_\odot(x)(R_\succ(y)+R_\prec(y))
 -R_\odot(y)R_\circ(x)\bigr)\otimes I
 \notag\\
&\quad
 +I\otimes\bigl(
 L_\odot(x)R_\prec(y)-R_\prec(x\odot y)
 \bigr)
 \notag\\
&\quad
 +L_\odot(x)\otimes(L_\star(y)-R_\odot(y))
 -(L_\star(y)-R_\odot(y))\otimes L_\odot(x)
 \Bigr](r+\tau(r)),
 \mlabel{eq:general-cb-k1}\\
0={}&\Bigl[
 \bigl(L_\star(x\succ y)-L_\succ(x)L_\star(y)
 +R_\odot(y)L_\circ(x)\bigr)\otimes I
 \notag\\
&\quad
 +I\otimes\bigl(
 L_\succ(x)R_\prec(y)-R_\prec(x\succ y)
 \bigr)
 \notag\\
&\quad
 +L_\succ(x)\otimes R_\prec(y)
 -(L_\star(y)-R_\odot(y))\otimes L_\succ(x)
 \Bigr](r+\tau(r)),
 \mlabel{eq:general-cb-k2}\\
0={}&\bigl[
 -R_\prec(x)L_\star(y)\otimes I
 +R_\prec(x)\otimes R_\prec(y)
 \bigr](r+\tau(r)).
 \mlabel{eq:general-cb-k3}
\end{align}
Consequently, if the three maps form an NS-Leibniz coalgebra, they define an NS-Leibniz bialgebra if and only if \meqref{eq:general-cb-k1}--\meqref{eq:general-cb-k3} hold.
\end{pro}

\begin{proof}
The coboundary formulas imply
\[
\begin{aligned}
\Delta_{\odot,r}(x)
 &=(L_\succ(x)\otimes I+I\otimes L_\circ(x))r,\\
\tau\Delta_{\prec,r}(x)
 &=(I\otimes L_\star(x)-R_\prec(x)\otimes I)r,\\
\Delta_r(x)
 &=-\bigl((L_\prec(x)+R_\succ(x))\otimes I
 +I\otimes R_\succ(x)\bigr)r\\
 &\quad
 +(L_\star(x)\otimes I-I\otimes R_\prec(x))(r+\tau(r)).
\end{aligned}
\]
Substituting these expressions into the left-hand side of \meqref{eq:k3}, we obtain
\[
\begin{aligned}
&(I\otimes R_\succ(y))\tau\Delta_{\prec,r}(x)
 -(R_\prec(x)\otimes I)\Delta_r(y)\\
={}&\bigl[
 R_\prec(x)R_\odot(y)\otimes I
 +I\otimes R_\succ(y)L_\star(x)
 \bigr]r\\
&+\bigl[
 -R_\prec(x)L_\star(y)\otimes I
 +R_\prec(x)\otimes R_\prec(y)
 \bigr](r+\tau(r)).
\end{aligned}
\]
The coefficient of $r$ vanishes by \meqref{eq:ns1}--\meqref{eq:ns3}. Hence \meqref{eq:k3} is equivalent to \meqref{eq:general-cb-k3}.

Substituting the same expressions into \meqref{eq:k1} and \meqref{eq:k2}, and collecting the terms containing $r+\tau(r)$, yields \meqref{eq:general-cb-k1} and \meqref{eq:general-cb-k2}. The coefficients of $r$ vanish by the NS-Leibniz identities.

For \meqref{eq:k4}, substitution yields
\[
\begin{aligned}
\Bigl[
&\bigl(L_\succ(x\circ y)
 -L_\succ(x)L_\succ(y)
 +L_\succ(y)L_\succ(x)\bigr)\otimes I\\
&+I\otimes\bigl(L_\circ(x\circ y)
 -L_\circ(x)L_\circ(y)
 +L_\circ(y)L_\circ(x)\bigr)
\Bigr]r=0.
\end{aligned}
\]
Similarly, \meqref{eq:k5} and \meqref{eq:k6} follow from the NS-Leibniz identities. The last assertion follows from Definition~\mref{def:bialgebra}.
\end{proof}

\begin{thm}\mlabel{thm:general-coboundary-criterion}
Let $(A,\succ,\prec,\vee)$ be an NS-Leibniz algebra and $r\in A\otimes A$. Then $(A,\succ,\prec,\vee,\Delta_{\succ,r},\Delta_{\prec,r},\Delta_{\vee,r})$ is an NS-Leibniz bialgebra if and only if \meqref{eq:general-cb-coalgebra-p}--\meqref{eq:general-cb-coalgebra-v} and \meqref{eq:general-cb-k1}--\meqref{eq:general-cb-k3} hold.
\end{thm}

\begin{proof}
This follows from Definition~\mref{def:bialgebra} and Propositions~\mref{prop:general-coboundary-coalgebra} and~\mref{prop:general-coboundary-compatibility}.
\end{proof}

\begin{pro}\mlabel{prop:skew-compatibility}
If $r\in A\otimes A$ is skew-symmetric, then the six compatibility conditions \meqref{eq:k1}--\meqref{eq:k6} hold. Consequently, the coboundary maps define an NS-Leibniz bialgebra if and only if they define an NS-Leibniz coalgebra.
\end{pro}

\begin{proof}
Since $r+\tau(r)=0$, the three equations in Proposition~\mref{prop:general-coboundary-compatibility} hold automatically. Hence all six compatibility conditions hold.
\end{proof}

For $r=\sum_i a_i\otimes b_i$, the tensor products used below are
\[
\begin{aligned}
r_{13}\diamond r_{23}
 &=\sum_{i,j}a_i\otimes a_j\otimes(b_i\diamond b_j),\\
r_{12}\diamond r_{23}
 &=\sum_{i,j}a_i\otimes(b_i\diamond a_j)\otimes b_j,\\
r_{21}\diamond r_{13}
 &=\sum_{i,j}(b_i\diamond a_j)\otimes a_i\otimes b_j,\\
r_{13}\diamond r_{21}
 &=\sum_{i,j}(a_i\diamond b_j)\otimes a_j\otimes b_i.
\end{aligned}
\]
Set
\begin{align}
S_\prec(r)
 &=r_{13}\prec r_{23}-r_{12}\prec r_{23}
 -r_{21}\circ r_{13}-r_{13}\circ r_{21},
 \mlabel{eq:ybe-p}\\
S_\succ(r)
 &=r_{13}\succ r_{23}+r_{12}\circ r_{23}
 +r_{21}\prec r_{13}+r_{13}\succ r_{21},
 \mlabel{eq:ybe-s}\\
S_\vee(r)
 &=r_{13}\vee r_{23}-r_{12}\vee r_{23}
 +r_{21}\vee r_{13}+r_{13}\vee r_{21}.
 \mlabel{eq:ybe-v}
\end{align}
Write $S_\circ(r)=S_\prec(r)+S_\succ(r)+S_\vee(r)$.

\begin{pro}\mlabel{thm:skew-coboundary-criterion}
Let $r\in A\otimes A$ be skew-symmetric. The maps \meqref{eq:delta-r-s}--\meqref{eq:delta-r-v} define an NS-Leibniz bialgebra if and only if the following equations hold for every $x\in A$:
\begin{align}
0={}&(I\otimes I\otimes R_\prec(x))S_\prec(r)
 \notag\\
&-(I\otimes L_\odot(x)\otimes I)
 (I\otimes\tau)S_\prec(r)
 \notag\\
&-(L_\star(x)\otimes I\otimes I)
 (\tau\otimes I)(I\otimes\tau)S_\circ(r),
 \mlabel{eq:skew-coalgebra-p}\\
0={}&(I\otimes I\otimes R_\prec(x))S_\succ(r)
 \notag\\
&+(L_\odot(x)\otimes I\otimes I)
 (\tau\otimes I)(I\otimes\tau)S_\prec(r)
 \notag\\
&+(I\otimes L_\star(x)\otimes I)
 (I\otimes\tau)S_\circ(r),
 \mlabel{eq:skew-coalgebra-sp}\\
0={}&(L_\odot(x)\otimes I\otimes I)
 (\tau\otimes I)(I\otimes\tau)S_\succ(r)
 \notag\\
&-(I\otimes L_\odot(x)\otimes I)
 (I\otimes\tau)S_\succ(r)
 \notag\\
&-(I\otimes I\otimes R_\circ(x))S_\circ(r),
 \mlabel{eq:skew-coalgebra-s}\\
0={}&(L_\odot(x)\otimes I\otimes I)
 (\tau\otimes I)(I\otimes\tau)S_\vee(r)
 \notag\\
&+((L_\vee(x)+R_\vee(x))\otimes I\otimes I)
 (\tau\otimes I)(I\otimes\tau)S_\circ(r)
 \notag\\
&+(I\otimes I\otimes R_\prec(x))S_\vee(r)
 +(I\otimes I\otimes R_\vee(x))S_\circ(r)
 \notag\\
&-(I\otimes L_\odot(x)\otimes I)
 (I\otimes\tau)S_\vee(r)
 \notag\\
&-(I\otimes(L_\vee(x)+R_\vee(x))\otimes I)
 (I\otimes\tau)S_\circ(r).
 \mlabel{eq:skew-coalgebra-v}
\end{align}
\end{pro}

\begin{proof}
By Proposition~\mref{prop:skew-compatibility}, it suffices to verify the coalgebra identities. Since $r+\tau(r)=0$, we have
\[
\Delta_r(x)
=-\sum_i\bigl[
 (x\prec a_i+a_i\succ x)\otimes b_i
 +a_i\otimes(b_i\succ x)
\bigr].
\]
Direct substitution yields
\[
\begin{aligned}
&(\Delta_r\otimes I)\Delta_{\succ,r}(x)
 -(I\otimes\Delta_{\succ,r})\Delta_{\succ,r}(x)
 +(\tau\otimes I)(I\otimes\Delta_{\succ,r})
 \Delta_{\succ,r}(x)\\
&=A(1)+A(2)+A(3),
\end{aligned}
\]
where
\begin{align*}
A(1)={}&\sum_{i,j}\Bigl[
 -(a_j\odot(x\odot a_i))\otimes b_j\otimes b_i
 -(x\odot a_i)\otimes(b_i\odot a_j)\otimes b_j\\
&\quad
 +(x\odot a_i)\otimes a_j\otimes(b_j\circ b_i)
 -((b_i\circ x)\odot a_j)\otimes a_i\otimes b_j
 \Bigr],\\
A(2)={}&\sum_{i,j}\Bigl[
 -a_j\otimes(b_j\succ(x\odot a_i))\otimes b_i
 +a_i\otimes((b_i\circ x)\odot a_j)\otimes b_j\\
&\quad
 +(b_i\odot a_j)\otimes(x\odot a_i)\otimes b_j
 -a_j\otimes(x\odot a_i)\otimes(b_j\circ b_i)
 \Bigr],\\
A(3)={}&\sum_{i,j}\Bigl[
 (a_i\prec a_j+a_j\succ a_i)\otimes b_j\otimes(b_i\circ x)
 +a_j\otimes(b_j\succ a_i)\otimes(b_i\circ x)\\
&\quad
 -a_i\otimes a_j\otimes(b_j\circ(b_i\circ x))
 +a_j\otimes a_i\otimes(b_j\circ(b_i\circ x))
 \Bigr].
\end{align*}
The first three NS-Leibniz identities imply
\begin{equation}\mlabel{eq:odot-computation}
\begin{aligned}
y\odot(x\odot z)-(y\circ x)\odot z
 &=x\odot(y\succ z),\\
y\odot(x\odot z)&=y\succ(x\odot z).
\end{aligned}
\end{equation}
Indeed, subtracting \meqref{eq:ns2} from \meqref{eq:ns1} with the first two variables exchanged yields $(x\odot z)\prec y=0$. Expanding $\odot$ and applying \meqref{eq:ns2} and \meqref{eq:ns3} then establishes \meqref{eq:odot-computation}.

Using these identities and $\tau(r)=-r$, we obtain
\begin{align*}
A(1)
={}&\sum_{i,j}\Bigl[
 -((a_j\circ x)\odot a_i)\otimes b_j\otimes b_i
 -(x\odot(a_j\succ a_i))\otimes b_j\otimes b_i\\
&\quad
 -(x\odot a_i)\otimes(b_i\odot a_j)\otimes b_j
 +(x\odot a_i)\otimes a_j\otimes(b_j\circ b_i)\\
&\quad
 -((b_j\circ x)\odot a_i)\otimes a_j\otimes b_i
 \Bigr]\\
={}&(L_\odot(x)\otimes I\otimes I)
 \sum_{i,j}\Bigl[
 -(a_j\succ a_i)\otimes b_j\otimes b_i\\
&\qquad
 -a_i\otimes(b_i\odot a_j)\otimes b_j
 +a_i\otimes a_j\otimes(b_j\circ b_i)
 \Bigr]\\
={}&-(L_\odot(x)\otimes I\otimes I)
 (\tau\otimes I)(I\otimes\tau)S_\succ(r).
\end{align*}
Similarly,
\begin{align*}
A(2)
={}&(I\otimes L_\odot(x)\otimes I)
 \sum_{i,j}\Bigl[
 -a_j\otimes(b_j\succ a_i)\otimes b_i\\
&\qquad
 +(b_i\odot a_j)\otimes a_i\otimes b_j
 -a_j\otimes a_i\otimes(b_j\circ b_i)
 \Bigr]\\
={}&(I\otimes L_\odot(x)\otimes I)
 (I\otimes\tau)S_\succ(r).
\end{align*}
Using the Leibniz identity for $\circ$, we obtain
\begin{align*}
A(3)
={}&(I\otimes I\otimes R_\circ(x))
 \sum_{i,j}\Bigl[
 (a_i\prec a_j+a_j\succ a_i)\otimes b_j\otimes b_i\\
&\qquad
 +a_j\otimes(b_j\succ a_i)\otimes b_i
 +a_i\otimes a_j\otimes(b_i\circ b_j)
 \Bigr]\\
={}&(I\otimes I\otimes R_\circ(x))S_\circ(r).
\end{align*}
Consequently, \meqref{eq:c3} is equivalent to \meqref{eq:skew-coalgebra-s}.

For the first two coalgebra identities, the same computation yields
\begin{align*}
&(I\otimes\Delta_r)\Delta_{\prec,r}(x)
 -(\Delta_{\prec,r}\otimes I)\Delta_{\prec,r}(x)\\
&\qquad
 -(\tau\otimes I)(I\otimes\Delta_{\prec,r})
 \Delta_{\succ,r}(x)\\
={}&(I\otimes I\otimes R_\prec(x))S_\prec(r)
 -(I\otimes L_\odot(x)\otimes I)
 (I\otimes\tau)S_\prec(r)\\
&\qquad
 -(L_\star(x)\otimes I\otimes I)
 (\tau\otimes I)(I\otimes\tau)S_\circ(r),
\end{align*}
and
\begin{align*}
&(I\otimes\Delta_{\prec,r})\Delta_{\succ,r}(x)
 -(\Delta_{\succ,r}\otimes I)\Delta_{\prec,r}(x)\\
&\qquad
 -(\tau\otimes I)(I\otimes\Delta_r)
 \Delta_{\prec,r}(x)\\
={}&(I\otimes I\otimes R_\prec(x))S_\succ(r)
 +(L_\odot(x)\otimes I\otimes I)
 (\tau\otimes I)(I\otimes\tau)S_\prec(r)\\
&\qquad
 +(I\otimes L_\star(x)\otimes I)
 (I\otimes\tau)S_\circ(r).
\end{align*}
Thus \meqref{eq:c1} and \meqref{eq:c2} are equivalent to \meqref{eq:skew-coalgebra-p} and \meqref{eq:skew-coalgebra-sp}, respectively.

Finally, substituting the coboundary maps into the difference between the two sides of \meqref{eq:c4}, and using \meqref{eq:ns1}--\meqref{eq:ns4}, yields
\begin{align*}
&(L_\odot(x)\otimes I\otimes I)
 (\tau\otimes I)(I\otimes\tau)S_\vee(r)\\
&+((L_\vee(x)+R_\vee(x))\otimes I\otimes I)
 (\tau\otimes I)(I\otimes\tau)S_\circ(r)\\
&-(I\otimes L_\odot(x)\otimes I)
 (I\otimes\tau)S_\vee(r)\\
&-(I\otimes(L_\vee(x)+R_\vee(x))\otimes I)
 (I\otimes\tau)S_\circ(r)\\
&+(I\otimes I\otimes R_\prec(x))S_\vee(r)
 +(I\otimes I\otimes R_\vee(x))S_\circ(r).
\end{align*}
Hence \meqref{eq:c4} is equivalent to \meqref{eq:skew-coalgebra-v}. This completes the proof.
\end{proof}

\begin{rmk}\mlabel{prop:skew-ybe-reduction}
Put
\[
 S(r)=r_{23}\circ r_{13}-r_{12}\odot r_{23}-r_{12}\succ r_{13}.
\]
If $r$ is skew-symmetric, the tensors obtained above satisfy
\begin{align}
 S_\circ(r)&=\tau_{12}S(r),&
 S_\succ(r)&=-\tau_{23}\tau_{12}S(r),\mlabel{eq:single-ybe-reduction-s}\\
 S_\prec(r)&=(\rho+\tau_{13})\tau_{12}S(r),&
 S_\vee(r)&=(I^{\o3}+\tau_{23}-\rho-\tau_{13})\tau_{12}S(r),
 \mlabel{eq:single-ybe-reduction-pv}
\end{align}
where $\tau_{13}(x\o y\o z)=z\o y\o x$.
Consequently, $S(r)=0$ if and only if
$S_\prec(r)=S_\succ(r)=S_\vee(r)=0$.
\end{rmk}
\begin{proof}
Summing \meqref{eq:ybe-p}--\meqref{eq:ybe-v} yields
\[
 S_\circ(r)=r_{13}\circ r_{23}+r_{12}\succ r_{23}
                -r_{21}\succ r_{13}-r_{13}\prec r_{21}.
\]
For skew-symmetric $r$,
\[
\begin{aligned}
\tau_{12}(r_{23}\circ r_{13})&=r_{13}\circ r_{23},\\
\tau_{12}(r_{12}\odot r_{23})
 &=r_{21}\succ r_{13}+r_{13}\prec r_{21},\\
\tau_{12}(r_{12}\succ r_{13})&=-r_{12}\succ r_{23}.
\end{aligned}
\]
Thus $S_\circ(r)=\tau_{12}S(r)$.
Expanding the tensor-leg products and using $\tau(r)=-r$
also yields
$S_\succ=-\tau_{23}S_\circ$ and
$S_\prec=(\rho+\tau_{13})S_\circ$.
Subtracting these two equalities from
$S_\circ=S_\prec+S_\succ+S_\vee$ proves the formula for
$S_\vee$. The equivalence follows, since $\tau_{12}$ is invertible.
\end{proof}

\begin{cor}\mlabel{cor:skew-tensor-zero}
If a skew-symmetric tensor $r\in A\o A$ satisfies $S(r)=0$,
then $S_\prec(r)=S_\succ(r)=S_\vee(r)=S_\circ(r)=0$.
\end{cor}
\begin{proof}
Apply the tensor-permutation relations in
Remark~\mref{prop:skew-ybe-reduction}.
\end{proof}

\begin{defi}\mlabel{def:ybe}
Let $A$ be an NS-Leibniz algebra and $r\in A\o A$.
The equation
\begin{equation}\mlabel{eq:ns-ybe}
 S(r)=r_{23}\circ r_{13}-r_{12}\odot r_{23}
                         -r_{12}\succ r_{13}=0
\end{equation}
is called the NS-Leibniz Yang--Baxter equation.
Here $\circ=\succ+\prec+\vee$ and
$x\odot y=x\succ y+y\prec x$.
The preceding calculation shows why this equation is sufficient
for the four coalgebra conditions when $r$ is skew-symmetric.
For $\vee=0$, it is the LD-YBE of
\mcite{SunGuo}, Definition~3.12.
\end{defi}

\begin{thm}\mlabel{thm:skew-coboundary}
Every skew-symmetric solution of \meqref{eq:ns-ybe}
produces an NS-Leibniz bialgebra with coproducts
\meqref{eq:delta-r-s}--\meqref{eq:delta-r-v}.
\end{thm}
\begin{proof}
By Corollary~\mref{cor:skew-tensor-zero},
$S_\prec(r)=S_\succ(r)=S_\vee(r)=S_\circ(r)=0$.
The four coalgebra conditions in
Proposition~\mref{thm:skew-coboundary-criterion} therefore hold.
The compatibility conditions hold by
Proposition~\mref{prop:skew-compatibility}.
\end{proof}

\subsection{\texorpdfstring{$\mathcal O$}{O}-operators and induced splittings}
\mlabel{sec:ooperators}
For $r=\sum_i a_i\o b_i\in A\o A$, define
\begin{equation}\mlabel{eq:rsharp}
 T_r(\zeta)=\sum_i \zeta(a_i)b_i,\qquad T_r^{\mathsf t}=T_{\tau(r)}.
\end{equation}
If $r$ is skew-symmetric, then $T_r^{\mathsf t}=-T_r$.

We now characterize solutions of the NS-Leibniz
Yang--Baxter equation by $\mathcal O$-operators and use
these operators to split the three products. Product
splittings also occur in quadri-algebras \mcite{AguiarLoday}.

\begin{defi}\mlabel{def:ooperator}
For a representation $(V,l_\diamond,r_\diamond)$ of $A$, a linear map
$T:V\to A$ is an $\mathcal{O}$-operator if, for every $\diamond\in\{\succ,\prec,\vee\}$,
\begin{equation}\mlabel{eq:ooperator}
 T(u)\diamond T(v)=T\bigl(l_\diamond(Tu)v+r_\diamond(Tv)u\bigr).
\end{equation}
For the regular representation, $T$ is a Rota--Baxter operator of
weight zero for each product.
\end{defi}

\begin{thm}\mlabel{thm:r-o}
Let $r\in A\o A$ be skew-symmetric. The following conditions
are equivalent:
\begin{enumerate}[label=(\alph*)]
\item $r$ satisfies the NS-Leibniz Yang--Baxter equation.
\item $T_r$ is an $\mathcal O$-operator on $A$ for the dual
regular representation; explicitly,
\begin{equation}\mlabel{eq:skew-contraction}
 T_r\zeta\diamond T_r\eta
 =T_r\bigl(l_\diamond(T_r\zeta)\eta
          +r_\diamond(T_r\eta)\zeta\bigr),
 \qquad \diamond\in\{\succ,\prec,\vee\}.
\end{equation}
\item $T_r$ is an $\mathcal O$-operator on $(A,\circ)$ for
$(A^*,L_\succ^*,-L_\odot^*)$; explicitly,
\begin{equation}\mlabel{eq:single-total-ooperator}
 T_r\zeta\circ T_r\eta
 =T_r\bigl(L_\succ^*(T_r\zeta)\eta
                   -L_\odot^*(T_r\eta)\zeta\bigr).
\end{equation}
\end{enumerate}
\end{thm}
\begin{proof}
Expanding the tensor-leg products and using $\tau(r)=-r$,
we obtain
\begin{equation}\mlabel{eq:quadratic-defect-contraction}
\begin{aligned}
&(\zeta\o\eta\o I)S_\diamond(r)\\
&\quad=T_r\zeta\diamond T_r\eta
 -T_r\bigl(l_\diamond(T_r\zeta)\eta
          +r_\diamond(T_r\eta)\zeta\bigr),
\qquad \diamond\in\{\succ,\prec,\vee\}.
\end{aligned}
\end{equation}
Here the left-hand side denotes contraction of the first
two tensor factors. For example,
\begin{align*}
\langle S_\prec(r),\zeta\o\eta\o h\rangle
={}&\sum_{i,j}\bigl[
 \zeta(a_i)\eta(a_j)h(b_i\prec b_j)
 -\zeta(a_i)\eta(b_i\prec a_j)h(b_j)\\
&\quad-\zeta(b_i\circ a_j)\eta(a_i)h(b_j)
 -\zeta(a_i\circ b_j)\eta(a_j)h(b_i)\bigr]\\
={}&\bigl\langle h,T_r\zeta\prec T_r\eta
 +T_r\bigl(L_\prec^*(T_r\zeta)\eta
            +L_\star^*(T_r\eta)\zeta\bigr)\bigr\rangle.
\end{align*}
The formulas for $\succ$ and $\vee$ follow by contracting
their displayed tensor expressions in the same order.
Thus condition (b) is equivalent to
$S_\prec=S_\succ=S_\vee=0$.
Remark~\mref{prop:skew-ybe-reduction} proves the equivalence
of (a) and (b). Summing the three contractions yields
\[
\begin{aligned}
&\big\langle h,T_r\zeta\circ T_r\eta
 -T_r\bigl(L_\succ^*(T_r\zeta)\eta
                   -L_\odot^*(T_r\eta)\zeta\bigr)\big\rangle\\
&\qquad=\langle S_\circ(r),\zeta\o\eta\o h\rangle
 =\langle\tau_{12}S(r),\zeta\o\eta\o h\rangle.
\end{aligned}
\]
Nondegeneracy of the pairing proves the equivalence with (c).
\end{proof}

The case of a tensor with invariant symmetric part is treated
in Section~\mref{sec:quasitriangular}.

\begin{rmk}
If $\vee=0$, then $S_\vee=0$. The other two equations are the
operator equations for the Leibniz-dendriform dual regular representation. Their skew-symmetric solutions are therefore
those of the LD-YBE in \mcite{SunGuo}.
\end{rmk}

\begin{thm}\mlabel{thm:nondegenerate-r}
Let $r\in A\o A$ be skew-symmetric and nondegenerate, and put
$\varpi(x,y)=\langle T_r^{-1}x,y\rangle$.
Then $r$ satisfies the NS-Leibniz Yang--Baxter equation
if and only if
\begin{equation}\mlabel{eq:skew-cocycle}
 \varpi(x,y\odot z)-\varpi(y,x\succ z)
                       +\varpi(z,x\circ y)=0
\end{equation}
for all $x,y,z\in A$.
\end{thm}
\begin{proof}
Since $T_r$ is invertible and $r$ is skew-symmetric,
$\varpi$ is nondegenerate and skew-symmetric. Set $x=T_r\zeta$ and
$y=T_r\eta$ in \meqref{eq:single-total-ooperator}.
Applying $T_r^{-1}$ and pairing with $z$ yields
\[
 \varpi(x\circ y,z)
 =-\varpi(y,x\succ z)+\varpi(x,y\odot z).
\]
By skew-symmetry, this is \meqref{eq:skew-cocycle}.
All steps are reversible, so Theorem~\mref{thm:r-o}
proves the equivalence.
\end{proof}

\begin{defi}\mlabel{def:skew-cocycle}
A bilinear form on an NS-Leibniz algebra is called a
$2$-cocycle if it satisfies \meqref{eq:skew-cocycle}.
\end{defi}

For the dual regular representation, the equation
\meqref{eq:ooperator} is \meqref{eq:skew-contraction}.
Thus, by Theorem~\mref{thm:r-o}, a skew-symmetric tensor $r$
satisfies the NS-Leibniz Yang--Baxter equation if and only if
$T_r$ is an $\mathcal{O}$-operator for this representation.

An $\mathcal O$-operator determines a skew-symmetric
tensor in the dual semidirect product. The following
theorem characterizes when this tensor satisfies the
NS-Leibniz Yang--Baxter equation.

\begin{thm}\mlabel{thm:T-tensor}
Let $V$ be a representation of $A$ and let
$\widehat A=A\ltimes V^*$ use its dual representation
\meqref{eq:dual-representation}.
Let $T:V\to A$ be a linear map. For a basis $(v_i)$ with dual
basis $(v_i^*)$, put
\begin{equation}\mlabel{eq:T-tensor}
 r_T=\sum_i\bigl(Tv_i\o v_i^*-v_i^*\o Tv_i\bigr).
\end{equation}
This tensor is basis-independent and skew-symmetric. It satisfies the
NS-Leibniz Yang--Baxter equation in $\widehat A$ if and only
if $T$ is an $\mathcal{O}$-operator for the original representation on $V$.
Consequently, every such operator produces a coboundary NS-Leibniz
bialgebra on $\widehat A$ by
\meqref{eq:delta-r-s}--\meqref{eq:delta-r-v}.
\end{thm}

\begin{proof}
The tensor $\sum_i Tv_i\otimes v_i^*$ corresponds to $T$ under $\operatorname{Hom}(V,A)\cong A\otimes V^*$. Hence $r_T$ is basis-independent and skew-symmetric. Under $\widehat A^*=A^*\oplus V$, its contraction map is
\[
T_{r_T}(\zeta,u)=(-Tu,T^{\mathsf t}\zeta).
\]

For $\diamond\in\{\succ,\prec,\vee\}$, put
\[
D_\diamond(u,v)
=Tu\diamond Tv
-T\bigl(l_\diamond(Tu)v+r_\diamond(Tv)u\bigr),
\]
and let $E_\diamond$ denote the difference between the two sides of \meqref{eq:skew-contraction} for $T_{r_T}$. By \meqref{eq:dual-representation} and the double duality in Proposition~\mref{thm:dual}, direct calculation yields
\[
E_\diamond(u,v)=(D_\diamond(u,v),0),
\qquad
E_\diamond(\zeta,\eta)=0.
\]
For mixed inputs, the $A$ components vanish. We identify $A^*$ and $V$ with the corresponding summands of $\widehat A^*$. For example, the dual actions associated with $\vee$ in \meqref{eq:dual-representation} yield, for $u,v\in V$ and $\zeta\in A^*$,
\[
\begin{aligned}
&\bigl\langle\bigl(E_\vee(u,\zeta)\bigr)_{V^*},v\bigr\rangle\\
&\quad=-\langle\zeta,T(l_\vee(Tu)v)\rangle
       -\langle L_\vee^*(Tu)\zeta,Tv\rangle
       -\langle\zeta,T(r_\vee(Tv)u)\rangle\\
&\quad=\bigl\langle\zeta,
       Tu\vee Tv-T\bigl(l_\vee(Tu)v+r_\vee(Tv)u\bigr)
       \bigr\rangle\\
&\quad=\langle\zeta,D_\vee(u,v)\rangle.
\end{aligned}
\]
Here the subscript $V^*$ denotes the $V^*$ component. The remaining mixed components are likewise linear combinations of the functionals
\[
v\longmapsto\langle\zeta,D_\diamond(u,v)\rangle,
\qquad
v\longmapsto\langle\zeta,D_\diamond(v,u)\rangle,
\qquad \diamond\in\{\succ,\prec,\vee\}.
\]
Thus all $E_\diamond$ vanish if and only if all $D_\diamond$ vanish. The equivalence follows from Theorem~\mref{thm:r-o}, and the final assertion follows from Theorem~\mref{thm:skew-coboundary}.
\end{proof}
The three induced products satisfy the NS-Leibniz
identities, as the following proposition shows.

\begin{pro}\mlabel{prop:induced}
An $\mathcal{O}$-operator $T:V\to A$ induces an NS-Leibniz algebra on $V$ by
\begin{equation}\mlabel{eq:induced}
 u\diamond_Tv=l_\diamond(Tu)v+r_\diamond(Tv)u,\qquad\diamond\in\{\succ,\prec,\vee\}.
\end{equation}
The map $T:(V,\succ_T,\prec_T,\vee_T)\to A$ is a homomorphism.
\end{pro}
\begin{proof}
Equation \meqref{eq:ooperator} states that
$\{(Tu,u):u\in V\}$ is a subalgebra of $A\ltimes V$.
Projection onto $V$ identifies its products with \meqref{eq:induced},
and projection onto $A$ is the homomorphism $T$.
\end{proof}

\begin{defi}\mlabel{def:six-splitting}
A six-operation splitting consists of bilinear operations
$\diamond^L,\diamond^R$, $\diamond\in\{\succ,\prec,\vee\}$, on $V$ such that the sums
$u\diamond v=u\diamond^Lv+u\diamond^Rv$ form an NS-Leibniz algebra and the maps
\[
 a_\diamond(u)v=u\diamond^Lv,\qquad b_\diamond(v)u=u\diamond^Rv
\]
form a representation of the sum algebra in the sense of
Definition~\mref{def:representation}.

\end{defi}

\begin{thm}\mlabel{thm:six-splitting}
An $\mathcal{O}$-operator induces a six-operation splitting by
\begin{equation}\mlabel{eq:six-induced}
 u\diamond^Lv=l_\diamond(Tu)v,\qquad
 u\diamond^Rv=r_\diamond(Tv)u.
\end{equation}
Conversely, for a six-operation splitting the identity map on $V$ is
an $\mathcal{O}$-operator of its sum algebra for the representation
$(a_\diamond,b_\diamond)$. If $T$ is invertible, the transported
operations on $A$ are
\[
 x\diamond^Ly=T\bigl(l_\diamond(x)T^{-1}y\bigr),\qquad
 x\diamond^Ry=T\bigl(r_\diamond(y)T^{-1}x\bigr).
\]
\end{thm}
\begin{proof}
By Proposition~\mref{prop:induced}, $T$ is a homomorphism for the sum
products. Pullback along $T$ therefore defines a representation on
$V$ with actions $l_\diamond\circ T,r_\diamond\circ T$.
These are the actions associated with \meqref{eq:six-induced}.
Conversely, the defining equality
$u\diamond v=a_\diamond(u)v+b_\diamond(v)u$ is the
$\mathcal{O}$-operator equation for the identity map. If $T$ is invertible,
transport of the operations through $T$ yields the stated formulas
on $A$.
\end{proof}

\begin{defi}\mlabel{def:split-invariant-form}
A bilinear form $\varpi$ on a six-operation splitting is
called invariant if, for all $x,y,z$,
\begin{align}
\varpi(x\succ^L y,z)&=-\varpi(y,x\circ z),&
\varpi(x\succ^R y,z)&=-\varpi(x,z\succ y+y\prec z),
\mlabel{eq:split-form-s}\\
\varpi(x\prec^L y,z)&=\varpi(y,x\prec z),&
\varpi(x\prec^R y,z)&=\varpi(x,y\circ z+z\circ y),
\mlabel{eq:split-form-p}\\
\varpi(x\vee^L y,z)&=\varpi(y,x\vee z),&
\varpi(x\vee^R y,z)&=-\varpi(x,y\vee z+z\vee y).
\mlabel{eq:split-form-v}
\end{align}
Here $\diamond=\diamond^L+\diamond^R$ for
$\diamond\in\{\succ,\prec,\vee\}$ and
$\circ=\succ+\prec+\vee$.
\end{defi}

\begin{pro}\mlabel{prop:cocycle-splitting}
Let $A$ be an NS-Leibniz algebra with a nondegenerate
skew-symmetric $2$-cocycle $\varpi$. Then there is a unique
six-operation splitting of its products satisfying
\meqref{eq:split-form-s}--\meqref{eq:split-form-v}.
Conversely, an invariant skew-symmetric form on such a
splitting is a $2$-cocycle of its sum NS-Leibniz algebra.
\end{pro}
\begin{proof}
Put $T=(\varpi^\sharp)^{-1}$.
Theorems~\mref{thm:nondegenerate-r} and~\mref{thm:r-o}
show that $T$ is an $\mathcal O$-operator for the dual
regular representation. Transporting the splitting of
Theorem~\mref{thm:six-splitting} through $T$ yields
\[
 x\diamond^L y=T(l_\diamond(x)T^{-1}y),\qquad
 x\diamond^R y=T(r_\diamond(y)T^{-1}x).
\]
Pairing with $z$ produces
\meqref{eq:split-form-s}--\meqref{eq:split-form-v}.
Nondegeneracy proves uniqueness. Conversely, summing the
six invariance equalities yields
\[
 \varpi(x\circ y,z)=-\varpi(y,x\succ z)+\varpi(x,y\odot z),
\]
which is \meqref{eq:skew-cocycle} by skew-symmetry.
\end{proof}

\begin{rmk}
If $\vee^L=\vee^R=0$, the identity map is an
$\mathcal O$-operator on the sum Leibniz-dendriform algebra.
By \mcite{SunGuo}, Proposition~3.24, the remaining four
operations form its corresponding Leibniz-quadri-algebra.
\end{rmk}

\begin{cor}\mlabel{cor:canonical-splitting-tensor}
For a six-operation splitting on $V$, its identity $\mathcal{O}$-operator
produces the skew tensor
$\sum_i(v_i\o v_i^*-v_i^*\o v_i)$ in the dual semidirect
product of the sum NS algebra and $V^*$. This tensor satisfies the
NS-Leibniz Yang--Baxter equation and defines a coboundary bialgebra there.
\end{cor}
\begin{proof}
Apply Theorem~\mref{thm:T-tensor} to the identity $\mathcal{O}$-operator
of Theorem~\mref{thm:six-splitting}.
\end{proof}

\section{Quasi-triangular and factorizable NS-Leibniz bialgebras}
\mlabel{sec:quasitriangular}
\subsection{Quasi-triangular NS-Leibniz bialgebras}
We now consider solutions of the NS-Leibniz Yang--Baxter
equation whose symmetric parts are invariant under the
three coboundary maps. Invariant symmetric tensors also
appear in the theory of Leibniz bialgebras
\mcite{BaiLiuShengTang}. When $\vee=0$, the invariance
conditions below reduce to those for Leibniz-dendriform
bialgebras \mcite{SunGuo}.

\begin{defi}\mlabel{def:invariant-tensor}
A tensor $q\in A\o A$ is invariant if, for every $x\in A$,
\begin{align}
 (L_\odot(x)\o I-I\o R_\circ (x))q&=0,\mlabel{eq:q-inv-s}\\
 (L_\star(x)\o I-I\o R_\prec(x))\tau(q)&=0,
 \mlabel{eq:q-inv-p}\\
 ((L_\vee(x)+R_\vee(x))\o I+I\o R_\vee(x))q&=0.
 \mlabel{eq:q-inv-v}
\end{align}
Equivalently, $\Delta_{\diamond,q}=0$ for
$\diamond\in\{\succ,\prec,\vee\}$. Below, $q$ denotes the symmetric tensor
$r+\tau(r)$.
\end{defi}

\begin{lem}\mlabel{lem:q-intertwining}
For a symmetric tensor $q$, invariance is equivalent to
\begin{align}
 R_\circ (x)T_q&=-T_q L_\odot^*(x),\mlabel{eq:q-contraction-s}\\
 R_\prec(x)T_q&=-T_q L_\star^*(x),\mlabel{eq:q-contraction-p}\\
 R_\vee(x)T_q&=T_q(L_\vee+R_\vee)^*(x).
 \mlabel{eq:q-contraction-v}
\end{align}
It is also equivalent to the following six intertwining relations,
where $l_\diamond,r_\diamond$ denote the dual regular actions:
\begin{equation}\mlabel{eq:T_q-intertwining}
 T_q l_\diamond(x)=L_\diamond(x)T_q,\qquad
 T_q r_\diamond(x)=R_\diamond(x)T_q,\qquad \diamond\in\{\succ,\prec,\vee\}.
\end{equation}
Under these conditions,
\begin{equation}\mlabel{eq:q-product}
\begin{aligned}
 C_\diamond(\zeta,\eta)&:=l_\diamond(T_q\zeta)\eta=r_\diamond(T_q\eta)\zeta,\\
 T_q C_\diamond(\zeta,\eta)&=T_q\zeta\diamond T_q\eta.
\end{aligned}
\end{equation}
\end{lem}
\begin{proof}
Contract \meqref{eq:q-inv-s}--\meqref{eq:q-inv-v} in the first
factor to obtain \meqref{eq:q-contraction-s}--\meqref{eq:q-contraction-v}.
Since $T_q^{\mathsf t}=T_q$, their transposes are
\[
 L_\odot(x)T_q=-T_q R_\circ ^*(x),\quad
 L_\star(x)T_q=-T_q R_\prec^*(x),\quad
 (L_\vee(x)+R_\vee(x))T_q=T_q R_\vee^*(x).
\]
Subtracting $R_\prec(x)T_q$ from $L_\odot(x)T_q$ yields
$L_\succ(x)T_q=T_qL_\circ^*(x)$. The other relations imply
\begin{align*}
 L_\vee(x)T_q&=-R_\vee(x)T_q+T_qR_\vee^*(x),\\
 L_\prec(x)T_q&=-R_\succ(x)T_q+T_qR_\succ^*(x).
\end{align*}
The right-hand sides are symmetric under ordinary transposition.
Hence $L_\vee(x)T_q=-T_qL_\vee^*(x)$ and
$L_\prec(x)T_q=-T_qL_\prec^*(x)$. We obtain
\begin{align*}
L_\succ(x)T_q&=T_q L_\circ ^*(x),&
R_\prec(x)T_q&=-T_q L_\star^*(x),\\
L_\vee(x)T_q&=-T_q L_\vee^*(x),&
R_\vee(x)T_q&=T_q(L_\vee+R_\vee)^*(x),\\
L_\prec(x)T_q&=-T_q L_\prec^*(x),&
R_\succ(x)T_q&=T_q(L_\prec+R_\succ)^*(x).
\end{align*}
These are \meqref{eq:T_q-intertwining}. Conversely, summing these
relations yields \meqref{eq:q-contraction-s}--\meqref{eq:q-contraction-v}.

Pairing $\Delta_{\diamond,q}=0$ with $\zeta\o\eta$ yields
$l_\diamond(T_q\zeta)\eta=r_\diamond(T_q\eta)\zeta$.
Applying $T_q$ and using \meqref{eq:T_q-intertwining} proves
\meqref{eq:q-product}.
\end{proof}

\begin{lem}\mlabel{lem:deformed-double}
Let $q$ be symmetric and invariant. The products
\begin{equation}\mlabel{eq:E-product}
 (x,\zeta)\diamond_E(y,\eta)=
 \bigl(x\diamond y,\ l_\diamond(x)\eta+r_\diamond(y)\zeta-C_\diamond(\zeta,\eta)\bigr)
\end{equation}
on $E_q=A\oplus A^*$ form an NS-Leibniz algebra. It is quadratic for
\begin{equation}\mlabel{eq:E-form}
 \omega_q((x,\zeta),(y,\eta))=\zeta(y)+\eta(x)-\langle \zeta,T_q\eta\rangle.
\end{equation}
Moreover, $(A^*,C_\succ,C_\prec,C_\vee)$ is an NS-Leibniz algebra.

\end{lem}
\begin{proof}
By \meqref{eq:q-product}, substitution of $C_\diamond$ into the NS
identities reduces them to \meqref{eq:r1}, \meqref{eq:r4},
\meqref{eq:r7} and \meqref{eq:r10}, with arguments
$T_q\zeta,T_q\eta$. Thus $C_\diamond$ defines an NS-Leibniz algebra.

For the products on $E_q$, triples with two entries in $A$ satisfy
the representation identities. For triples with two entries in
$A^*$, use \meqref{eq:T_q-intertwining} to move $T_q$ through the
actions. The resulting identities are again the representation
relations, with one argument of the form $T_q\zeta$.
Each term contains one internal product, so the common minus sign
cancels. For example, the fourth identity on $x,\zeta,\eta$ becomes
\begin{align*}
 &l_\succ(x)l_\vee(T_q\zeta)\eta+l_\vee(x)l_\circ (T_q\zeta)\eta\\
 &\quad=l_\prec(x\vee T_q\zeta)\eta+l_\vee(x\circ T_q\zeta)\eta
       +l_\succ(T_q\zeta)l_\vee(x)\eta+l_\vee(T_q\zeta)l_\circ (x)\eta,
\end{align*}
which is \meqref{eq:r10}; the other positions correspond to the
right-action relations. Triples in either summand have already
been treated, proving the NS identities on $E_q$.

Symmetry of $T_q$ implies symmetry of $\omega_q$.
If $(x,\zeta)$ is orthogonal to $A\oplus A^*$, pairing
first with $(y,0)$ yields $\zeta=0$, and then pairing with
$(0,\eta)$ yields $x=0$. Hence $\omega_q$ is nondegenerate.
Substitution of \meqref{eq:E-product} into
\meqref{eq:inv-p} and \meqref{eq:inv-s} establishes invariance:
the semidirect-product terms cancel by the dual regular
representation, and the terms involving $T_q$ and
$C_\diamond$ cancel by \meqref{eq:T_q-intertwining} and
\meqref{eq:q-product}.
\end{proof}

\begin{pro}\mlabel{prop:non-skew-products}
Let $r\in A\o A$, and let $T_r,T_{\tau(r)}:A^*\to A$ be its contraction maps.
The products dual to its three coboundary maps are
\begin{align}
\zeta\succ_r \eta&=L_\circ ^*(T_r\zeta)\eta-R_\odot^*(T_{\tau(r)}\eta)\zeta,
 \mlabel{eq:non-skew-s}\\
\zeta\prec_r \eta&=L_\prec^*(T_{\tau(r)}\zeta)\eta-L_\star^*(T_r\eta)\zeta,
 \mlabel{eq:non-skew-p}\\
\zeta\vee_r \eta&=-L_\vee^*(T_r\zeta)\eta-(L_\vee+R_\vee)^*(T_{\tau(r)}\eta)\zeta.
 \mlabel{eq:non-skew-v}
\end{align}
If $q=r+\tau(r)$ is invariant, they have the uniform expression
\begin{equation}\mlabel{eq:non-skew-uniform}
 \zeta\diamond_r \eta=l_\diamond(T_r\zeta)\eta+r_\diamond(T_r\eta)\zeta-C_\diamond(\zeta,\eta).
\end{equation}
\end{pro}
\begin{proof}
Pair \meqref{eq:delta-r-s}--\meqref{eq:delta-r-v} with
$\zeta\o\eta$ to obtain \meqref{eq:non-skew-s}--\meqref{eq:non-skew-v}.
If $q$ is invariant, substitute $T_{\tau(r)}=T_q-T_r$.
The additional terms are $-l_\prec(T_q\zeta)\eta$ for $\prec$
and $-r_\diamond(T_q\eta)\zeta$ for $\succ,\vee$.
Equation \meqref{eq:q-product} identifies them with
$-C_\diamond(\zeta,\eta)$.
\end{proof}

Let $q=r+\tau(r)$ be invariant. Define the auxiliary tensors
\begin{align}
 H_\prec(r)&=S_\prec(r)+q_{12}\prec r_{23},
 \mlabel{eq:H-p}\\
 H_\succ(r)&=S_\succ(r)-q_{12}\circ r_{23},
 \mlabel{eq:H-s}\\
 H_\vee(r)&=S_\vee(r)+q_{12}\vee r_{23}.
 \mlabel{eq:H-v}
\end{align}
Here $S_\diamond(r)$ denotes the tensor expression in
\meqref{eq:ybe-p}--\meqref{eq:ybe-v} for an arbitrary $r$. If $q=\sum_i a_i\o b_i$ and
$r=\sum_j u_j\o v_j$, then
$q_{12}\diamond r_{23}=\sum_{i,j}a_i\o(b_i\diamond u_j)\o v_j$.

\begin{lem}\mlabel{lem:single-non-skew-equation}
Assume that $q=r+\tau(r)$ is invariant. Then
\begin{equation}\mlabel{eq:H-total-single}
 H_\prec(r)+H_\succ(r)+H_\vee(r)=\tau_{12}S(r).
\end{equation}
Moreover, $S(r)=0$ if and only if
$H_\prec(r)=H_\succ(r)=H_\vee(r)=0$.
\end{lem}
\begin{proof}
Summing \meqref{eq:H-p}--\meqref{eq:H-v} yields
$S_\circ(r)-q_{12}\succ r_{23}$.
Direct expansion of $\tau_{12}S(r)$ yields the same tensor,
which proves \meqref{eq:H-total-single} without any
symmetry assumption on $r$.

For the converse implication, contraction identifies
$H_\diamond(r)$ with
\[
 T_r\zeta\diamond T_r\eta
 -T_r\bigl(l_\diamond(T_r\zeta)\eta
          +r_\diamond(T_r\eta)\zeta-C_\diamond(\zeta,\eta)\bigr).
\]
If $S(r)=0$, their sum vanishes, so the graph
$G_r=\{(T_r\zeta,\zeta):\zeta\in A^*\}$ is a Leibniz
subalgebra of $E_q$. By \meqref{eq:E-form},
\[
 \omega_q((T_r\zeta,\zeta),(T_r\eta,\eta))
 =\zeta(T_r\eta)+\eta(T_r\zeta)-\zeta(T_q\eta)=0.
\]
Thus $G_r^\perp=G_r$. For $a,b,c\in G_r$, invariance implies
\[
 \omega_q(a\succ b,c)=-\omega_q(b,a\circ c)=0,\qquad
 \omega_q(a\prec b,c)=\omega_q(a,b\circ c+c\circ b)=0.
\]
Hence $a\succ b,a\prec b\in G_r$; subtraction from
$a\circ b$ yields $a\vee b\in G_r$.
The graph is therefore an NS-Leibniz subalgebra, and each
$H_\diamond(r)$ vanishes. The other implication follows
immediately from \meqref{eq:H-total-single}.
\end{proof}

\begin{defi}\mlabel{def:non-skew-ybe}
A coboundary NS-Leibniz bialgebra is called quasi-triangular
if its defining tensor $r$ satisfies \meqref{eq:ns-ybe}
and $r+\tau(r)$ is invariant. It is called triangular if,
in addition, $r$ is skew-symmetric.
\end{defi}

\begin{thm}
\mlabel{thm:quasitriangular}
Assume $q=r+\tau(r)$ is invariant. The following are equivalent:
\begin{enumerate}[label=(\roman*)]
\item $S(r)=0$.
\item $T_r$ satisfies, for each of the three operations,
\begin{equation}\mlabel{eq:weighted-T_r}
 T_r\zeta\diamond T_r\eta=T_r\bigl(l_\diamond(T_r\zeta)\eta+r_\diamond(T_r\eta)\zeta-C_\diamond(\zeta,\eta)\bigr).
\end{equation}
\item The graph $\{(T_r\zeta,\zeta):\zeta\in A^*\}$ is an NS-Leibniz subalgebra
of $E_q$.
\item The dual products \meqref{eq:non-skew-s}--\meqref{eq:non-skew-v}
form an NS-Leibniz algebra, and both
$T_r$ and $-T_{\tau(r)}$ are NS-Leibniz homomorphisms from it to $A$.
\end{enumerate}
Under these conditions, \meqref{eq:delta-r-s}--\meqref{eq:delta-r-v}
define an NS-Leibniz bialgebra.
\end{thm}
\begin{proof}
By Lemma~\mref{lem:single-non-skew-equation}, (i) is
equivalent to $H_\prec(r)=H_\succ(r)=H_\vee(r)=0$.
The contraction calculation in Theorem~\mref{thm:r-o}
identifies $S_\diamond(r)$ with the defect in
\meqref{eq:skew-contraction}.
The defect in \meqref{eq:weighted-T_r} has the additional term
$T_rC_\diamond(\zeta,\eta)$. Its pairings are
\[
 \langle h,T_r C_\prec(\zeta,\eta)\rangle=\langle q_{12}\prec r_{23},\zeta\o \eta\o h\rangle,\quad
 \langle h,T_r C_\succ(\zeta,\eta)\rangle=-\langle q_{12}\circ r_{23},\zeta\o \eta\o h\rangle,
\]
\[
 \langle h,T_r C_\vee(\zeta,\eta)\rangle=\langle q_{12}\vee r_{23},\zeta\o \eta\o h\rangle.
\]
Nondegeneracy of the pairing proves the equivalence of (i) and
(ii). By \meqref{eq:E-product}, (ii) is the closure condition for
the graph in (iii), whose induced products are
\meqref{eq:non-skew-uniform}.

Under (ii), \meqref{eq:T_q-intertwining} implies
\begin{equation}\mlabel{eq:T_q-dual-product}
 T_q(\zeta\diamond_r \eta)=T_r\zeta\diamond T_q\eta+T_q\zeta\diamond T_r\eta-T_q\zeta\diamond T_q\eta.
\end{equation}
Subtracting from
$T_r(\zeta\diamond_r\eta)=T_r\zeta\diamond T_r\eta$ yields
\[
 \bigl(-T_{\tau(r)}\bigr)(\zeta\diamond_r\eta)
 =\bigl(-T_{\tau(r)}\bigr)\zeta\diamond\bigl(-T_{\tau(r)}\bigr)\eta.
\]
Thus (iv) holds. Conversely, the homomorphism condition for $T_r$
in (iv) is (ii).

The graph is complementary to $A\oplus0$ and satisfies
\[
 \omega_q((T_r\zeta,\zeta),(T_r\eta,\eta))=\zeta(T_r\eta)+\eta(T_r\zeta)-\zeta(T_q\eta)=0.
\]
Both summands are isotropic subalgebras of the quadratic algebra
$E_q$, so they form a Manin triple. Their pairing is the evaluation
pairing. By Proposition~\mref{prop:non-skew-products}, the coproducts
dual to the graph products are
\meqref{eq:delta-r-s}--\meqref{eq:delta-r-v}.
Theorem~\mref{thm:bialgebra} now applies.
\end{proof}

\begin{rmk}
For $q=0$, $H_\diamond(r)=S_\diamond(r)$ and the theorem reduces
to Theorem~\mref{thm:skew-coboundary}. When $\vee=0$, the theorem reduces to the
corresponding characterization of quasi-triangular
Leibniz-dendriform bialgebras \mcite{SunGuo}. For general $q$, the correction terms in
\meqref{eq:H-p}--\meqref{eq:H-v} correspond to the internal products
$C_\diamond$ in \meqref{eq:weighted-T_r}.
\end{rmk}

\begin{pro}\mlabel{prop:transpose-r}
If $r$ determines a quasi-triangular NS-Leibniz bialgebra, so does
$\tau(r)$. Their coproducts are negatives of one another, and their
symmetric parts are identical.
\end{pro}
\begin{proof}
Since $\Delta_{\diamond,q}=0$,
$\Delta_{\diamond,\tau(r)}=-\Delta_{\diamond,r}$.
The dual products for $\tau(r)$ are therefore $-\diamond_r$.
By Theorem~\mref{thm:quasitriangular},
$T_{\tau(r)}$ and $-T_r$ are homomorphisms from this algebra to $A$.
Applying the same theorem to $\tau(r)$ proves the proposition.
\end{proof}

We next consider the case in which the symmetric part is
nondegenerate.
\subsection{Factorizable NS-Leibniz bialgebras}
Factorizability is characterized by the nondegeneracy of
the symmetric part of the defining tensor.

\begin{defi}\mlabel{def:factorizable}
A quasi-triangular NS-Leibniz bialgebra associated with $r$ is
factorizable if $T_q=T_{r+\tau(r)}:A^*\to A$ is invertible.
\end{defi}

\begin{pro}\mlabel{prop:factorization-image}
For a factorizable structure, the map
\begin{equation}\mlabel{eq:factorization-image}
 F:A^*\longrightarrow A\oplus A,\qquad F(\zeta)=(T_r\zeta,-T_{\tau(r)}\zeta)
\end{equation}
is an injective NS-Leibniz homomorphism. For every $x\in A$ there is
exactly one pair $(x_1,x_2)$ in its image such that $x=x_1-x_2$;
namely
\[
 x_1=T_r T_q^{-1}x,\qquad x_2=\bigl(-T_{\tau(r)}\bigr) T_q^{-1}x.
\]
Consequently $A\oplus A=\operatorname{diag}(A)\oplus F(A^*)$ as vector
spaces, and both summands are NS-Leibniz subalgebras.
\end{pro}
\begin{proof}
Theorem~\mref{thm:quasitriangular} shows that $F$ is a homomorphism.
If $F(\zeta)=0$, then $T_q\zeta=0$; invertibility of $T_q$ yields
$\zeta=0$. The difference of the two coordinates of $F(\zeta)$ is
$T_q\zeta$. It follows that the difference map is a bijection from
$F(A^*)$ to $A$, with inverse $x\mapsto F(T_q^{-1}x)$.
The vector-space decomposition follows, and both summands are
subalgebras.
\end{proof}

The factorization identifies the canonical double with
the direct sum of two copies of $A$, as follows.

\begin{thm}\mlabel{thm:factorizable-double}
Let $D=A\bowtie A^*$ be the canonical NS double of a factorizable
structure. Then
\begin{equation}\mlabel{eq:factorizable-Phi}
 \Phi:D\longrightarrow A\oplus A,\qquad
 \Phi(x,\zeta)=(x+T_r\zeta,x-T_{\tau(r)}\zeta)
\end{equation}
is an NS-Leibniz isomorphism. Its inverse is
\[
 \zeta=T_q^{-1}(a-b),\qquad x=a-T_rT_q^{-1}(a-b).
\]
The symmetric form
\begin{equation}\mlabel{eq:B-S}
 \omega(x,y)=\langle T_q^{-1}x,y\rangle
\end{equation}
makes $A$ quadratic, and $\Phi$ identifies the canonical form on $D$
with $\omega\oplus(-\omega)$.
\end{thm}
\begin{proof}
The form $\omega$ is symmetric and nondegenerate because $T_q$ is.
By \meqref{eq:T_q-intertwining}, it is invariant. Define
\[
 \Psi:E_q\longrightarrow A\oplus A,\qquad
 \Psi(a,\zeta)=(a,a-T_q\zeta).
\]
Since $T_q$ is invertible, $\Psi$ is bijective. Equations
\meqref{eq:T_q-intertwining} and \meqref{eq:q-product} show that the
second coordinate of a product is
\[
 a\diamond b-a\diamond T_q\eta-T_q\zeta\diamond b+T_q\zeta\diamond T_q\eta
 =(a-T_q\zeta)\diamond(b-T_q\eta).
\]
Thus $\Psi$ is an NS-Leibniz isomorphism.
The identification of the canonical double with $E_q$ is
$(x,\zeta)\mapsto(x+T_r\zeta,\zeta)$; composition with $\Psi$
is $\Phi$. Solving for $\zeta,x$ yields the inverse in the statement.
Moreover,
\[
 \omega(a,b)-\omega(a-T_q\zeta,b-T_q\eta)
 =\zeta(b)+\eta(a)-\zeta(T_q\eta)=\omega_q((a,\zeta),(b,\eta)).
\]
The preceding identification pulls $\omega_q$ back to the
canonical form on $D$, proving the equality of forms.
\end{proof}

\begin{thm}\mlabel{thm:factorizable-operator}
Factorizable quasi-triangular tensors on $A$ correspond to pairs
$(\omega,P)$, where $\omega$ makes $A$ a quadratic NS-Leibniz
algebra and $P:A\to A$ satisfies, for each $\diamond\in\{\succ,\prec,\vee\}$,
\begin{equation}\mlabel{eq:weight-minus-one}
 Px\diamond Py=P(Px\diamond y+x\diamond Py-x\diamond y),
\end{equation}
and its $\omega$-adjoint satisfies
\begin{equation}\mlabel{eq:P-adjoint}
 P+P^\dagger=I.
\end{equation}
For a tensor $r$, take $\omega$ as in \meqref{eq:B-S} and set
$P=T_rT_q^{-1}$.
Conversely, set $T_q=(\omega^\sharp)^{-1}$ and let $r$ be the unique tensor with
$T_r=PT_q$. Then $r+\tau(r)$ is the inverse-form tensor of $\omega$, is
invariant, and $r$ is factorizable quasi-triangular.
\end{thm}
\begin{proof}
The isomorphism $T_q:A^*\to A$ carries $C_\diamond$ to the
products on $A$ by \meqref{eq:q-product}. Under this isomorphism,
\meqref{eq:weighted-T_r} becomes \meqref{eq:weight-minus-one}
for $P=T_rT_q^{-1}$. The equality
$T_r+T_{\tau(r)}=T_q$ is equivalent to $P+P^\dagger=I$.

Conversely, let $(\omega,P)$ satisfy the stated conditions.
Lemma~\mref{lem:q-intertwining} implies invariance of the inverse-form
tensor. The preceding identities, read in reverse, show that the
tensor with $T_r=P(\omega^\sharp)^{-1}$ has this symmetric part
and satisfies \meqref{eq:weighted-T_r}.
Theorem~\mref{thm:quasitriangular} completes the correspondence.
\end{proof}

The canonical double carries a natural factorizable
quasi-triangular NS-Leibniz bialgebra structure.

\begin{thm}\mlabel{thm:canonical-double-r}
Let $(A,\succ,\prec,\vee,\Delta_\succ,\Delta_\prec,\Delta_\vee)$ be
any NS-Leibniz bialgebra, with canonical double $D=A\bowtie A^*$.
For a basis $(e_i)$ of $A$ and its dual basis $(e_i^*)$, the tensor
\begin{equation}\mlabel{eq:canonical-double-r}
 r_D=\sum_i e_i\o e_i^*\in D\o D
\end{equation}
is basis-independent and defines a factorizable quasi-triangular
NS-Leibniz bialgebra on $D$. Its three coproducts restrict on the
subalgebra $A$ to the original three coproducts.
\end{thm}
\begin{proof}
The tensor $r_D$ corresponds to the identity of $A$ and is
independent of the dual bases. For the canonical form $\omega_D$,
$T_{q_D}=(\omega_D^\sharp)^{-1}$, where $q_D=r_D+\tau(r_D)$.
Lemma~\mref{lem:q-intertwining} implies that $q_D$ is invariant.

Let $P$ be projection onto $A^*$ along $A$. Both summands are
subalgebras, so for $z=a+\zeta$ and $w=b+\eta$,
\[
 Pz\diamond w+z\diamond Pw-z\diamond w=\zeta\diamond \eta-a\diamond b,
 \qquad P(\zeta\diamond \eta-a\diamond b)=\zeta\diamond \eta=Pz\diamond Pw.
\]
Thus $P$ is Rota--Baxter of weight $-1$. Isotropy of the summands
implies $P^\dagger=I-P$. Under $D^*=A^*\oplus A$,
\[
 T_{q_D}(\zeta,x)=(x,\zeta),\qquad
 T_{r_D}(\zeta,x)=(0,\zeta),\qquad T_{\tau(r_D)}(\zeta,x)=(x,0).
\]
Hence $T_{r_D}=PT_{q_D}$, and
Theorem~\mref{thm:factorizable-operator} proves factorizability.

For $x\in A$, substitute $r_D$ in
\meqref{eq:delta-r-s}--\meqref{eq:delta-r-v} on $D$.
The dual regular action formulas cancel the terms in
$A\o A^*$ and $A^*\o A$.
The remaining terms are
\begin{align*}
 \Delta_{\succ,r_D}(x)
 &=-\sum_i e_i\o L_{\succ,A^*}^*(e_i^*)x,\\
 \Delta_{\prec,r_D}(x)
 &=-\sum_i R_{\prec,A^*}^*(e_i^*)x\o e_i,\\
 \Delta_{\vee,r_D}(x)
 &=-\sum_i e_i\o L_{\vee,A^*}^*(e_i^*)x.
\end{align*}
For example, the $A^*$ component of $e_i^*\circ_D x$ is
$(L_{\odot,A}(x))^{\mathsf t}e_i^*$, which cancels
$(L_{\odot,A}(x)\o I)r_D$ in the first formula.
Pairing the three remaining expressions with $\zeta\o\eta$
yields, respectively,
\[
 \langle \zeta\succ\eta,x\rangle,\qquad
 \langle \zeta\prec\eta,x\rangle,\qquad
 \langle \zeta\vee\eta,x\rangle.
\]
Thus their values are the original coproducts by
\meqref{eq:coalgebra-pair}.
\end{proof}

\begin{ex}
Let $A$ be the algebra in Example~\mref{ex:metric}, and let
$D=A\ltimes A^*$ be its dual regular semidirect product.
The tensor $r_D=e\o e^*+f\o f^*$ is factorizable
quasi-triangular. Its coproducts vanish on $e,f$, and their nonzero
values on $e^*,f^*$ are
\begin{align*}
 \Delta_{\prec,r_D}(e^*)&=f^*\o f^*,&
 \Delta_{\prec,r_D}(f^*)&=f^*\o e^*,\\
 \Delta_{\succ,r_D}(f^*)&=-e^*\o f^*,&
 \Delta_{\vee,r_D}(e^*)&=-f^*\o f^*,\\
 \Delta_{\vee,r_D}(f^*)&=2e^*\o f^*-f^*\o e^*.
\end{align*}
The table is obtained from \meqref{eq:delta-r-s}--\meqref{eq:delta-r-v}
using the dual regular actions on $D$.
\end{ex}

\section{Relative Rota--Baxter operators and quadratic Rota--Baxter NS-Leibniz algebras}
\mlabel{sec:relative-rb}
\subsection{Relative Rota--Baxter operators and the Yang--Baxter equation}
Relative Rota--Baxter operators relate representations
with internal products to the NS-Leibniz Yang--Baxter equation.

\begin{defi}\mlabel{def:A-ns-algebra}
An $A$-NS-Leibniz algebra is an NS-Leibniz algebra
$(V,\succ_V,\prec_V,\vee_V)$ equipped with a representation
$(l_\diamond,r_\diamond)_{\diamond\in\{\succ,\prec,\vee\}}$ of $A$ such that
\begin{equation}\mlabel{eq:module-algebra}
 (x,u)\diamond(y,v)=
 (x\diamond y,\ l_\diamond(x)v+r_\diamond(y)u+u\diamond_V v)
\end{equation}
defines an NS-Leibniz algebra on $A\oplus V$.
A relative Rota--Baxter operator of weight $\lambda\in k$
associated with this $A$-NS-Leibniz algebra is a linear map $T:V\to A$
satisfying
\begin{equation}\mlabel{eq:relative-rb}
 T(u)\diamond T(v)=T\bigl(l_\diamond(T(u))v+r_\diamond(T(v))u
                         +\lambda u\diamond_V v\bigr),
 \qquad \diamond\in\{\succ,\prec,\vee\}.
\end{equation}
When the three internal products vanish, this is an $\mathcal{O}$-operator.
For $V=A$, with its regular actions and original products, it is a
Rota--Baxter operator on $A$ of weight $\lambda$.
\end{defi}

Summing the products in \meqref{eq:module-algebra} defines an $A$-Leibniz algebra. A relative Rota--Baxter NS-Leibniz operator is therefore a relative Rota--Baxter operator for the total products. Conversely, let $r\in A\o A$ have invariant symmetric part $q=r+\tau(r)$. For $T=T_r$ and weight $-1$, Theorem~\mref{thm:relative-ybe} establishes the converse for the $A$-NS-Leibniz algebra on $A^*$ in Proposition~\mref{prop:invariant-module-algebra}.

The invariant symmetric tensor determines compatible
internal products on $A^*$ for the three dual actions.

\begin{pro}\mlabel{prop:invariant-module-algebra}
Let $q\in A\o A$ be symmetric and invariant. Define products on
$A^*$ by
\begin{align}
 \zeta\succeq_q\eta
   &=L_\circ^*(T_q\zeta)\eta
     =(L_\prec+R_\succ)^*(T_q\eta)\zeta,\mlabel{eq:internal-s}\\
 \zeta\preceq_q\eta
   &=-L_\prec^*(T_q\zeta)\eta
     =-L_\star^*(T_q\eta)\zeta,\mlabel{eq:internal-p}\\
 \zeta\vee_q\eta
   &=-L_\vee^*(T_q\zeta)\eta
     =(L_\vee+R_\vee)^*(T_q\eta)\zeta.\mlabel{eq:internal-v}
\end{align}
These products, together with the dual regular actions
\meqref{eq:dual-regular}, define an $A$-NS-Leibniz algebra.
Its total product is
\begin{equation}\mlabel{eq:internal-total}
 \zeta\circ_q\eta=L_\succ^*(T_q\zeta)\eta
                  =-(L_\succ+R_\prec)^*(T_q\eta)\zeta.
\end{equation}
In particular, the total product and total dual regular actions form
an $A$-Leibniz algebra.
\end{pro}
\begin{proof}
By \meqref{eq:dual-regular} and \meqref{eq:q-product}, the products
are $C_\diamond$. Applying Lemma~\mref{lem:deformed-double} to
$-q$, and using $C_\diamond^{-q}=-C_\diamond^q$, yields
\[
 (x,\zeta)\diamond(y,\eta)=
 (x\diamond y,l_\diamond(x)\eta+r_\diamond(y)\zeta+C_\diamond(\zeta,\eta)),
\]
which is \meqref{eq:module-algebra}. Summing the products and using
\meqref{eq:dual-total-partial} yields \meqref{eq:internal-total}.
\end{proof}

\begin{thm}\mlabel{thm:relative-ybe}
Let $r\in A\o A$ and assume that $q=r+\tau(r)$ is invariant.
The following statements are equivalent:
\begin{enumerate}[label=(\alph*)]
\item $r$ satisfies the NS-Leibniz Yang--Baxter equation.
The coboundary maps \meqref{eq:delta-r-s}--\meqref{eq:delta-r-v}
then define a quasi-triangular NS-Leibniz bialgebra.
\item $T_r$ is a relative Rota--Baxter operator of weight $-1$ for
 the $A$-NS-Leibniz algebra in
 Proposition~\mref{prop:invariant-module-algebra}.
\item $T_r$ is a relative Rota--Baxter operator of weight $-1$ for
 its total $A$-Leibniz algebra; explicitly,
\begin{equation}\mlabel{eq:relative-total}
 T_r\zeta\circ T_r\eta
 =T_r\bigl(L_\succ^*(T_r\zeta)\eta
 -(L_\succ+R_\prec)^*(T_r\eta)\zeta-\zeta\circ_q\eta\bigr).
\end{equation}
\end{enumerate}
\end{thm}
\begin{proof}
The equivalence of (a) and (b) is
Theorem~\mref{thm:quasitriangular}. Summing the operator identities
in (b) proves (c).

Suppose (c) holds. Contracting
\meqref{eq:H-total-single} shows that $S(r)=0$.
Lemma~\mref{lem:single-non-skew-equation} then implies
$H_\prec=H_\succ=H_\vee=0$, which is (b).
When $\vee=0$, these operator conditions reduce to those
for Leibniz-dendriform algebras \mcite{SunGuo}.
\end{proof}

\begin{rmk}
If $r$ is skew-symmetric, the products on $A^*$ in
Proposition~\mref{prop:invariant-module-algebra} vanish, and the
relative Rota--Baxter equation reduces to the $\mathcal{O}$-operator equation.
\end{rmk}

We now relate factorizable tensors to Rota--Baxter operators
on quadratic NS-Leibniz algebras.

\subsection{Quadratic Rota--Baxter NS-Leibniz algebras}
The relation between factorizability and quadratic
Rota--Baxter algebras is known in the Lie
\mcite{LangSheng}, pre-Lie \mcite{WangBaiLiuSheng}, and
Leibniz \mcite{BaiLiuShengTang} settings. We establish its
NS-Leibniz version below.
\begin{defi}\mlabel{def:quadratic-rb}
A quadratic Rota--Baxter NS-Leibniz algebra of weight $\lambda$ is
$(A,\succ,\prec,\vee,P,\omega)$, where $(A,\succ,\prec,\vee,\omega)$
is quadratic, $P$ is a Rota--Baxter operator of weight $\lambda$
for all three operations, and
\begin{equation}\mlabel{eq:rb-form}
 \omega(Px,y)+\omega(x,Py)+\lambda\omega(x,y)=0.
\end{equation}
Equivalently, $P+P^\dagger=-\lambda I$, where $P^\dagger$ is the
ordinary $\omega$-adjoint.
A Rota--Baxter Leibniz algebra with a compatible symmetric form is
$(A,\circ,P,\omega)$, where $\omega$ is nondegenerate and symmetric,
$P$ is Rota--Baxter of weight $\lambda$ for $\circ$, and
\meqref{eq:rb-form} holds.
\end{defi}

\begin{pro}\mlabel{prop:rb-complement}
In either setting of Definition~\mref{def:quadratic-rb}, replacing
$P$ by $-\lambda I-P$ preserves the structure and its weight.
\end{pro}
\begin{proof}
Put $Q=-\lambda I-P$ and
$W=Px\diamond y+x\diamond Py+\lambda x\diamond y$.
The Rota--Baxter identity for $P$ implies
\[
 Qx\diamond Qy=\lambda^2x\diamond y
 +\lambda(Px\diamond y+x\diamond Py)+P(W)
 =Q(-W).
\]
Since $Qx\diamond y+x\diamond Qy+\lambda x\diamond y=-W$,
$Q$ is Rota--Baxter of weight $\lambda$. The form condition follows
from
\[
 \omega(Qx,y)+\omega(x,Qy)+\lambda\omega(x,y)
 =-\omega(Px,y)-\omega(x,Py)-\lambda\omega(x,y)=0.
\]
The substitution $P\mapsto-\lambda I-P$ is an involution.
\end{proof}

A Rota--Baxter operator on the total Leibniz algebra
determines compatible operators for the three products
through the invariant bilinear form.

\begin{thm}\mlabel{thm:rb-metric-splitting}
Let $(A,\circ,P,\omega)$ be a Rota--Baxter Leibniz algebra with a
compatible symmetric form of weight $\lambda$.
The operations \meqref{eq:metric-construct} make
$(A,\succ,\prec,\vee,P,\omega)$ a quadratic Rota--Baxter NS-Leibniz
algebra of weight $\lambda$. Conversely, summing the products of any
quadratic Rota--Baxter NS-Leibniz algebra produces such a
Rota--Baxter Leibniz algebra. These constructions are inverse.
\end{thm}
\begin{proof}
Theorem~\mref{thm:metric-splitting} defines the quadratic
NS-Leibniz products. Denote the Rota--Baxter defects by
\[
 D_\diamond(x,y)=Px\diamond Py
 -P(Px\diamond y+x\diamond Py+\lambda x\diamond y).
\]
The form condition is
$\omega(Pu,v)=-\omega(u,(P+\lambda I)v)$.
Using this equality and \meqref{eq:metric-construct}, we obtain
\begin{align}
 \omega(D_\succ(x,y),z)&=-\omega(y,D_\circ(x,z)),\mlabel{eq:rb-defect-s}\\
 \omega(D_\prec(x,y),z)&=\omega(x,D_\circ(y,z)+D_\circ(z,y)).
 \mlabel{eq:rb-defect-p}
\end{align}
Indeed, for $\succ$,
\begin{align*}
 \omega(D_\succ(x,y),z)
 &=\omega\bigl(y,(P+\lambda I)(Px\circ z)
                  +P(x\circ(P+\lambda I)z)-Px\circ(P+\lambda I)z\bigr)\\
 &=-\omega(y,D_\circ(x,z)).
\end{align*}
For $\prec$, we obtain
\begin{align*}
 \omega(D_\prec(x,y),z)
 &=\omega\bigl(x,-(P+\lambda I)(Py\circ z+z\circ Py)\bigr)\\
 &\quad+\omega\bigl(x,-P(y\circ(P+\lambda I)z+(P+\lambda I)z\circ y)\bigr)\\
 &\quad+\omega\bigl(x,Py\circ(P+\lambda I)z+(P+\lambda I)z\circ Py\bigr)\\
 &=\omega(x,D_\circ(y,z)+D_\circ(z,y)).
\end{align*}
Since $D_\circ=0$, nondegeneracy of $\omega$ yields
$D_\succ=D_\prec=0$. Also
$D_\vee=D_\circ-D_\succ-D_\prec=0$.
Conversely, summing the NS Rota--Baxter identities yields the
Leibniz Rota--Baxter identity, while the form condition is unchanged.
Uniqueness in Theorem~\mref{thm:metric-splitting} proves that the
correspondences are inverse.
\end{proof}

\begin{rmk}\mlabel{rem:rb-symplectic-specialization}
By \meqref{eq:defect}, $\omega$ is symplectic for $\circ$ if and
only if $\vee=0$. Theorem~\mref{thm:rb-metric-splitting} then reduces
to the correspondence between Rota--Baxter symplectic Leibniz
algebras and quadratic Rota--Baxter Leibniz-dendriform algebras.
\end{rmk}

For a nondegenerate form $\omega$, write
$\langle \omega^\sharp(x),y\rangle=\omega(x,y)$ and define its inverse tensor
$r_\omega$ by
\begin{equation}\mlabel{eq:inverse-form-tensor}
 T_{r_\omega}=(\omega^\sharp)^{-1}.
\end{equation}

\begin{lem}\mlabel{lem:inverse-invariant-form}
For an NS-Leibniz algebra $A$ and a nondegenerate bilinear form
$\omega$, the form makes $A$ quadratic if and only if $r_\omega$ is
symmetric and invariant.
\end{lem}
\begin{proof}
Symmetry of $\omega$ is equivalent to symmetry of $r_\omega$.
For symmetric $r_\omega$, Lemma~\mref{lem:q-intertwining}
characterizes invariance by the intertwining relations for
$T_{r_\omega}$. Their inverses are the intertwining relations for
$\omega^\sharp$.
\end{proof}

\begin{pro}\mlabel{prop:rb-tensor-transport}
Let $(A,\succ,\prec,\vee,\omega)$ be quadratic and assume that
$q=r+\tau(r)$ is invariant. Put $P=T_r\omega^\sharp$.
Then $r$ satisfies the NS-Leibniz Yang--Baxter equation
if and only if
\begin{equation}\mlabel{eq:rb-tensor-transport}
 Px\diamond Py=P\bigl(Px\diamond y+x\diamond Py
                   -x\diamond T_q\omega^\sharp(y)\bigr),
 \qquad \diamond\in\{\succ,\prec,\vee\}.
\end{equation}
The correction also satisfies
$x\diamond T_q\omega^\sharp(y)=T_q\omega^\sharp(x)\diamond y$.
\end{pro}
\begin{proof}
Set $\zeta=\omega^\sharp(x)$ and $\eta=\omega^\sharp(y)$.
By the intertwining relations and \meqref{eq:q-product},
\begin{align*}
 l_\diamond(T_r\zeta)\eta&=\omega^\sharp(Px\diamond y),\\
 r_\diamond(T_r\eta)\zeta&=\omega^\sharp(x\diamond Py),\\
 C_\diamond(\zeta,\eta)
 &=\omega^\sharp(x\diamond T_q\omega^\sharp(y))
  =\omega^\sharp(T_q\omega^\sharp(x)\diamond y).
\end{align*}
Substitute in \meqref{eq:weighted-T_r} and apply
Theorem~\mref{thm:quasitriangular}.
\end{proof}

\begin{lem}\mlabel{lem:rb-symmetric-part}
For a nondegenerate symmetric form $\omega$, a tensor $r$, and
$P=T_r\omega^\sharp$, the equality
\begin{equation}\mlabel{eq:rb-symmetric-part}
 r+\tau(r)=-\lambda r_\omega
\end{equation}
is equivalent to \meqref{eq:rb-form}.
\end{lem}
\begin{proof}
For $\zeta=\omega^\sharp(x)$ and $\eta=\omega^\sharp(y)$,
\[
 \omega(Px,y)+\omega(x,Py)+\lambda\omega(x,y)
 =\langle r+\tau(r)+\lambda r_\omega,\zeta\o\eta\rangle.
\]
Since $\omega^\sharp$ is surjective, the right-hand side vanishes
for all $x,y$ if and only if \meqref{eq:rb-symmetric-part} holds.
\end{proof}

\begin{cor}\mlabel{cor:rb-zero-triangular}
Every quadratic Rota--Baxter NS-Leibniz algebra of weight zero
produces a triangular NS-Leibniz bialgebra on $A$, with tensor
$T_r=P(\omega^\sharp)^{-1}$ and the three coboundary maps
\meqref{eq:delta-r-s}--\meqref{eq:delta-r-v}.
\end{cor}
\begin{proof}
By Lemma~\mref{lem:rb-symmetric-part}, $r$ is skew-symmetric.
Proposition~\mref{prop:rb-tensor-transport} implies the tensor
equations with $q=0$. Apply Theorem~\mref{thm:skew-coboundary}.
\end{proof}

We now establish the correspondence between quadratic
Rota--Baxter NS-Leibniz algebras and factorizable
NS-Leibniz bialgebras.

\begin{thm}\mlabel{thm:rb-factorizable-weight}
Fix $\lambda\ne0$. Factorizable quasi-triangular tensors $r$ on $A$
are in one-to-one correspondence with quadratic Rota--Baxter
NS-Leibniz structures $(P,\omega)$ of weight $\lambda$ on $A$.
For a tensor $r$, the corresponding data are
\begin{equation}\mlabel{eq:rb-factorizable-weight}
 \omega(x,y)=-\lambda\langle T_{r+\tau(r)}^{-1}x,y\rangle,
 \qquad P=T_r\omega^\sharp.
\end{equation}
Conversely, the tensor associated with $(P,\omega)$ is determined by
\begin{equation}\mlabel{eq:rb-factorizable-converse}
 T_r=P(\omega^\sharp)^{-1},\qquad r+\tau(r)=-\lambda r_\omega.
\end{equation}
The correspondence concerns bialgebras equipped with their defining tensor.
\end{thm}
\begin{proof}
For a factorizable tensor, let $q=r+\tau(r)$.
Theorem~\mref{thm:factorizable-double} implies that
$\omega(x,y)=-\lambda\langle T_q^{-1}x,y\rangle$ is a nondegenerate invariant
symmetric form. Since $\omega^\sharp=-\lambda T_q^{-1}$,
$T_q\omega^\sharp=-\lambda I$.
Proposition~\mref{prop:rb-tensor-transport} then yields the
Rota--Baxter identity of weight $\lambda$.
Also $r_\omega=-\lambda^{-1}q$, so
Lemma~\mref{lem:rb-symmetric-part} yields \meqref{eq:rb-form}.

Conversely, let $(P,\omega)$ be quadratic Rota--Baxter of weight
$\lambda$, and define $r$ by \meqref{eq:rb-factorizable-converse}.
By Lemmas~\mref{lem:inverse-invariant-form} and
\mref{lem:rb-symmetric-part}, its symmetric part is invariant and
\[
 T_q=-\lambda(\omega^\sharp)^{-1}.
\]
The map is invertible because $\lambda\ne0$.
Proposition~\mref{prop:rb-tensor-transport} implies
$H_\diamond(r)=0$, and Theorem~\mref{thm:quasitriangular} proves
that $r$ is factorizable. The formulas in the statement are inverse
to one another.
\end{proof}

The dual regular representation provides the following
cotangent construction of quadratic Rota--Baxter
NS-Leibniz algebras.

\begin{pro}\mlabel{prop:rb-cotangent-lift}
Let $P$ be a Rota--Baxter operator of weight $\lambda$ on an
NS-Leibniz algebra $A$. On the dual regular semidirect product
$D=A\ltimes A^*$, define
\begin{equation}\mlabel{eq:rb-cotangent-lift}
 \widehat P(x,\zeta)=(Px,-P^{\mathsf t}\zeta-\lambda\zeta),
 \qquad
 \omega_D((x,\zeta),(y,\eta))=\zeta(y)+\eta(x).
\end{equation}
Here $P^{\mathsf t}$ denotes the ordinary transpose. Then $(D,\widehat P,\omega_D)$ is
a quadratic Rota--Baxter NS-Leibniz algebra of weight $\lambda$.
For dual bases $(e_i),(e_i^*)$, its tensor is
\begin{equation}\mlabel{eq:rb-cotangent-tensor}
 r=\sum_i\bigl(e_i^*\o P(e_i)
             -(P+\lambda I)(e_i)\o e_i^*\bigr).
\end{equation}
It yields a triangular NS-Leibniz bialgebra when $\lambda=0$,
and a factorizable quasi-triangular one when $\lambda\ne0$.
\end{pro}
\begin{proof}
By Theorem~\mref{thm:double}, $D$ is quadratic for $\omega_D$.
Put $Q=-P^{\mathsf t}-\lambda I$.
For $a=L_\diamond$ or $R_\diamond$, the Rota--Baxter identity reads
\begin{equation}\mlabel{eq:rb-regular-action}
 a(Px)P=P a(Px)+P a(x)P+\lambda P a(x).
\end{equation}
Taking transposes yields
\begin{equation}\mlabel{eq:rb-dual-action}
 a^*(Px)Q=Q\bigl(a^*(Px)+a^*(x)Q+\lambda a^*(x)\bigr).
\end{equation}
Writing $M=a(Px)^{\mathsf t}$, $N=a(x)^{\mathsf t}$ and
$B=P^{\mathsf t}$, the transposed equation is
\[
 BM=MB+BNB+\lambda NB.
\]
Hence
\[
 (-B-\lambda I)(-M+NB)=MB+\lambda M,
\]
which proves \meqref{eq:rb-dual-action}.
The same relation holds for each dual regular action, since these
are linear combinations of the $a^*$.
The $A$ component of the Rota--Baxter identity for $\widehat P$
is the identity for $P$; the $A^*$ component is
\meqref{eq:rb-dual-action} applied to the left and right actions.
There are no internal products on $A^*$.
Moreover,
\begin{align*}
 &\omega_D(\widehat P(x,\zeta),(y,\eta))
 +\omega_D((x,\zeta),\widehat P(y,\eta))
 +\lambda\omega_D((x,\zeta),(y,\eta))\\
 &=\zeta(Py)+Q\zeta(y)+\lambda\zeta(y)
   +\eta(Px)+Q\eta(x)+\lambda\eta(x)=0.
\end{align*}
Therefore the form condition holds. Contraction of
\meqref{eq:rb-cotangent-tensor} yields
$T_r=\widehat P(\omega_D^\sharp)^{-1}$, and
$r+\tau(r)=-\lambda\sum_i(e_i\o e_i^*+e_i^*\o e_i)$.
Corollary~\mref{cor:rb-zero-triangular} treats $\lambda=0$;
Theorem~\mref{thm:rb-factorizable-weight} treats $\lambda\ne0$.
\end{proof}

\begin{pro}\mlabel{prop:rb-transposed-tensor}
Under the correspondence of Theorem~\mref{thm:rb-factorizable-weight},
transposing a factorizable tensor $r$ leaves its quadratic form
unchanged and replaces $P$ by $-\lambda I-P$.
\end{pro}
\begin{proof}
The symmetric part of $r$ is unchanged by transposition, and
\[
 T_{\tau(r)}\omega^\sharp
 =(T_q-T_r)\omega^\sharp=-\lambda I-P.
\]
Apply Proposition~\mref{prop:rb-complement} and
Theorem~\mref{thm:rb-factorizable-weight}.
Since $\Delta_{\diamond,q}=0$, the coproducts for $\tau(r)$ are
$-\Delta_{\diamond,r}$.
\end{proof}

\begin{ex}\mlabel{ex:rb-canonical-specialization}
The zero map on any NS-Leibniz algebra is Rota--Baxter of every
weight. For $P=0$, Proposition~\mref{prop:rb-cotangent-lift} yields
$\widehat P(x,\zeta)=(0,-\lambda\zeta)$ and
$r=-\lambda\sum_i e_i\o e_i^*$.
For $\lambda=-1$, this is the canonical tensor on $A\ltimes A^*$.
For $A$ as in Example~\mref{ex:metric} and $\lambda\ne0$,
both the third product and the third coproduct are nonzero.
\end{ex}

\section{Nijenhuis operators and compatible NS-Leibniz bialgebras}
\mlabel{sec:nijenhuis}
We now apply Nijenhuis operators to the preceding constructions.
For the corresponding Leibniz theory, see
\mcite{MondalSaha,MaSunZheng}.

\begin{defi}\mlabel{def:nijenhuis}
For a bilinear product $\diamond$ and a linear map $N:A\to A$, put
\begin{equation}\mlabel{eq:nijenhuis-product}
 x\diamond_N y=Nx\diamond y+x\diamond Ny-N(x\diamond y).
\end{equation}
The map $N$ is Nijenhuis for this product if
\begin{equation}\mlabel{eq:nijenhuis-torsion}
 Nx\diamond Ny=N(x\diamond_N y).
\end{equation}
A Nijenhuis operator on an NS-Leibniz algebra satisfies this condition
for each of $\succ,\prec,\vee$.
For the total Leibniz algebra, the condition is imposed on $\circ$.
\end{defi}

\begin{pro}\mlabel{prop:rb-nijenhuis-obstruction}
If $P$ is Rota--Baxter of weight $\lambda$ on an NS-Leibniz algebra,
then its Nijenhuis torsion for each product is
\begin{equation}\mlabel{eq:rb-nijenhuis-obstruction}
 Px\diamond Py-P(x\diamond_P y)=P(P+\lambda I)(x\diamond y).
\end{equation}
Consequently, $P$ is Nijenhuis for all three operations if and only if
$P(P+\lambda I)$ annihilates
\[
 A\succ A+A\prec A+A\vee A,
\]
where $A\diamond A$ denotes the span of all products.
In particular, $P^2=-\lambda P$ is sufficient; it is also necessary
when these product spaces span $A$.
\end{pro}
\begin{proof}
The Rota--Baxter identity implies
\begin{align*}
 Px\diamond Py-P(Px\diamond y+x\diamond Py-P(x\diamond y))
 &=P(\lambda x\diamond y+P(x\diamond y)).
\end{align*}
Thus the torsion vanishes if and only if $P(P+\lambda I)$ vanishes
on the span of the products. If that span is $A$, this is equivalent
to $P^2=-\lambda P$.
\end{proof}

\begin{thm}
\mlabel{thm:metric-nijenhuis}
Let $(A,\succ,\prec,\vee,\omega)$ be quadratic. Suppose that $N$ is
Nijenhuis on $(A,\circ)$ and is skew-adjoint for $\omega$:
\begin{equation}\mlabel{eq:nijenhuis-skew}
 \omega(Nx,y)+\omega(x,Ny)=0.
\end{equation}
Then $(A,\succ_N,\prec_N,\vee_N,\omega)$ is a quadratic NS-Leibniz
algebra. More generally, for every $a,b\in k$, the products
\begin{equation}\mlabel{eq:ns-compatible-products}
 x\diamond_{a,b}y=a(x\diamond y)+b(x\diamond_N y)
\end{equation}
form a quadratic NS-Leibniz algebra for the same form $\omega$.
In particular, the original and deformed NS-Leibniz algebras are
compatible.
\end{thm}
\begin{proof}
Put $\circ_t=\circ+t\circ_N$. The Nijenhuis identity yields
\begin{equation}\mlabel{eq:nijenhuis-formal-map}
 (I+tN)(x\circ_t y)=(I+tN)x\circ(I+tN)y.
\end{equation}
Over $k[[t]]$, $I+tN$ is invertible, so $\circ_t$ is a Leibniz
product. Its Leibniz defect is quadratic in $t$. Vanishing of its
coefficients proves compatibility of $\circ$ and $\circ_N$,
and hence the Leibniz identity for $a\circ+b\circ_N$.

Skew-adjointness of $N$ and \meqref{eq:inv-p}--\meqref{eq:inv-s} imply
\begin{align}
 \omega(x\succ_N y,z)&=-\omega(y,x\circ_N z),\mlabel{eq:nijenhuis-metric-s}\\
 \omega(x\prec_N y,z)&=\omega(x,y\circ_N z+z\circ_N y).
 \mlabel{eq:nijenhuis-metric-p}
\end{align}
For $\succ_N$, the left-hand side expands to
\[
 -\omega(y,Nx\circ z)+\omega(y,N(x\circ z))
 -\omega(y,x\circ Nz),
\]
and for $\prec_N$, it expands to
\[
 \omega\bigl(x,Ny\circ z+y\circ Nz-N(y\circ z)
                  +Nz\circ y+z\circ Ny-N(z\circ y)\bigr).
\]
Since $\circ_N=\succ_N+\prec_N+\vee_N$,
Theorem~\mref{thm:metric-splitting} identifies these products with
the quadratic splitting of $\circ_N$.
By linearity, \meqref{eq:ns-compatible-products} is the quadratic
splitting of $a\circ+b\circ_N$ for every $a,b\in k$.
\end{proof}

For a quadratic Rota--Baxter NS-Leibniz algebra, the preceding
deformation also induces a deformation of its coboundary maps.

\begin{thm}\mlabel{thm:nijenhuis-bialgebra-pencil}
Let $(A,\succ,\prec,\vee,P,\omega)$ be a quadratic Rota--Baxter
NS-Leibniz algebra of weight $\lambda$. Let $N$ satisfy the hypotheses
of Theorem~\mref{thm:metric-nijenhuis}, and assume
\begin{equation}\mlabel{eq:PN-commuting}
 PN=NP.
\end{equation}
For all $a,b\in k$, $P$ is Rota--Baxter of weight $\lambda$ for the
products \meqref{eq:ns-compatible-products}.
The tensor
\begin{equation}\mlabel{eq:nijenhuis-fixed-tensor}
 T_r=P(\omega^\sharp)^{-1},\qquad r+\tau(r)=-\lambda r_\omega
\end{equation}
defines an NS-Leibniz bialgebra for every pair $(a,b)$.
Its coproducts are
\begin{equation}\mlabel{eq:nijenhuis-coproduct-pencil}
 \Delta_{\diamond,r}^{a,b}
 =a\Delta_{\diamond,r}+b\Delta_{\diamond,r}^{N},
 \qquad \diamond\in\{\succ,\prec,\vee\},
\end{equation}
where $\Delta_{\diamond,r}^{N}$ is the coboundary map for
$\diamond_N$. These bialgebras are triangular for $\lambda=0$ and
factorizable quasi-triangular for $\lambda\ne0$.
In particular, the original and deformed products and coproducts
constitute a compatible pair of NS-Leibniz bialgebra structures.
\end{thm}
\begin{proof}
For $u,v\in A$, put
$W(u,v)=Pu\circ v+u\circ Pv+\lambda u\circ v$.
Using $PN=NP$ and the Rota--Baxter identity for $P$, we obtain
\begin{align*}
 Px\circ_N Py
 &=NPx\circ Py+Px\circ NPy-N(Px\circ Py)\\
 &=P\bigl(W(Nx,y)+W(x,Ny)-N W(x,y)\bigr)\\
 &=P(Px\circ_N y+x\circ_N Py+\lambda x\circ_N y).
\end{align*}
Thus $P$ is Rota--Baxter of weight $\lambda$ for $\circ_N$,
and by linearity for $a\circ+b\circ_N$.
The form condition \meqref{eq:rb-form} is unchanged.
Theorems~\mref{thm:metric-nijenhuis} and
\mref{thm:rb-metric-splitting} imply the NS Rota--Baxter identities.

The maps $P$ and $\omega^\sharp$ determine $r$ independently
of $a,b$.
For $\lambda=0$, apply Corollary~\mref{cor:rb-zero-triangular};
for $\lambda\ne0$, apply Theorem~\mref{thm:rb-factorizable-weight}.
The coproduct formulas are linear in the regular actions, so their
values for $a\diamond+b\diamond_N$ are
\meqref{eq:nijenhuis-coproduct-pencil}. Every simultaneous linear
combination of the products and coproducts is therefore a bialgebra.
\end{proof}

\begin{pro}\mlabel{prop:nijenhuis-dual-deformation}
Under the preceding hypotheses,
\begin{equation}\mlabel{eq:nijenhuis-tensor-skew}
 (N\o I+I\o N)r=0,
\end{equation}
and all three deformed coproducts satisfy
\begin{equation}\mlabel{eq:nijenhuis-coproduct-explicit}
 \Delta_{\diamond,r}^{N}
 =\Delta_{\diamond,r}N
  -(N\o I+I\o N)\Delta_{\diamond,r}.
\end{equation}
If $\diamond^*$ denotes the product on $A^*$ dual to
$\Delta_{\diamond,r}$, its deformed product is
\begin{equation}\mlabel{eq:nijenhuis-dual-product}
 \zeta\diamond_N^*\eta
 =(-N^{\mathsf t}\zeta)\diamond^*\eta
  +\zeta\diamond^*(-N^{\mathsf t}\eta)
  +N^{\mathsf t}(\zeta\diamond^*\eta).
\end{equation}
Thus the induced dual deformation has operator $-N^{\mathsf t}$.
\end{pro}
\begin{proof}
Skew-adjointness yields
$N(\omega^\sharp)^{-1}=-(\omega^\sharp)^{-1}N^{\mathsf t}$.
Together with $PN=NP$, this implies $NT_r=-T_rN^{\mathsf t}$,
or equivalently \meqref{eq:nijenhuis-tensor-skew}.
For a left or right regular action,
\[
 a_{\diamond,N}(x)=a_\diamond(Nx)+a_\diamond(x)N-Na_\diamond(x).
\]
Put $K=N\o I+I\o N$ and denote by $D_\diamond(x)$
the tensor action in the corresponding coboundary formula.
Then
\[
 D_{\diamond,N}(x)=D_\diamond(Nx)+D_\diamond(x)K-KD_\diamond(x).
\]
For $\succ,\vee$, apply this equality to $r$; for $\prec$, apply
it to $\tau(r)$. Since $Kr=K\tau(r)=0$, we obtain
\meqref{eq:nijenhuis-coproduct-explicit}.
Pairing with $\zeta\o\eta$ proves
\meqref{eq:nijenhuis-dual-product}.
\end{proof}

\begin{cor}\mlabel{cor:nijenhuis-dual-total}
If $\lambda\ne0$ in Theorem~\mref{thm:nijenhuis-bialgebra-pencil},
then $-N^{\mathsf t}$ is Nijenhuis for the original total Leibniz
product on $A^*$. If $N$ is additionally Nijenhuis for each original
NS operation on $A$, then $-N^{\mathsf t}$ is Nijenhuis for each of
the three original dual operations as well.
\end{cor}
\begin{proof}
Since $\lambda\ne0$, the tensor is factorizable.
Proposition~\mref{prop:factorization-image} yields an injective
homomorphism $F(\zeta)=(T_r\zeta,-T_{\tau(r)}\zeta)$.
Equation \meqref{eq:nijenhuis-tensor-skew} implies
$F(-N^{\mathsf t}\zeta)=(N\oplus N)F(\zeta)$.
Applying $F$ to the dual Nijenhuis torsion yields the
torsion of $N\oplus N$ on $F(A^*)$. It vanishes for the total
product, so injectivity proves the first statement.
If $N$ is Nijenhuis for each NS product, the same argument applies
to the corresponding dual product.
\end{proof}

Projections associated with Manin triples satisfy the hypotheses
of the preceding results.

\begin{pro}\mlabel{prop:manin-nijenhuis-pencil}
Let $(D,U,V,\omega)$ be an NS-Leibniz Manin triple and let $P$ be
projection onto $V$ along $U$. Then $P$ is quadratic Rota--Baxter
of weight $-1$, and $N=I-2P$ is skew-adjoint and Nijenhuis for all
three products. The preceding construction yields a compatible
factorizable family on $D$, with tensor $T_r=P(\omega^\sharp)^{-1}$.
Writing $\diamond_t=\diamond+t\diamond_N$, its products restrict to
\begin{equation}\mlabel{eq:manin-pencil-restrictions}
 u\diamond_t u'=(1+t)(u\diamond u'),\qquad
 v\diamond_t v'=(1-t)(v\diamond v'),
 \quad u,u'\in U,\ v,v'\in V.
\end{equation}
\end{pro}
\begin{proof}
As in Theorem~\mref{thm:canonical-double-r}, $P$ is Rota--Baxter
of weight $-1$, and isotropy yields $P^\dagger=I-P$.
Hence $N^\dagger=-N$. Since $P^2=P$,
Proposition~\mref{prop:rb-nijenhuis-obstruction} implies that $P$
is Nijenhuis. For any bilinear product, the torsion of $aI+bP$
is $b^2$ times the torsion of $P$; taking $a=1,b=-2$ proves that
$N$ is Nijenhuis. Also $PN=NP$.
Apply Theorem~\mref{thm:nijenhuis-bialgebra-pencil} with $(a,b)=(1,t)$.
The restrictions \meqref{eq:manin-pencil-restrictions} follow from
$N|_U=I$, $N|_V=-I$ and closure of the two summands.
\end{proof}

\begin{ex}
\mlabel{ex:nontrivial-nijenhuis-pencil}
Let $B$ be the four-dimensional factorizable dual regular double
constructed from Example~\mref{ex:metric}, with basis
$(e,f,e^*,f^*)$ and canonical tensor $e\o e^*+f\o f^*$.
Let $(e^\vee,f^\vee,E,F)$ be its dual basis and let
$D=B\bowtie B^*$ be the canonical double of this bialgebra.
In $B$ one has $e\succ f=f$ and $f\vee f=e$.
From the coproduct table in Section~\mref{sec:quasitriangular}, the dual algebra satisfies
\[
 F\prec F=E,\qquad F\vee F=-E.
\]
For example, the first equality follows by pairing $F\o F$
with $\Delta_\prec(e^*)=f^*\o f^*$; its pairings with the
other three basis elements are zero.
Apply Proposition~\mref{prop:manin-nijenhuis-pencil} to $(D,B,B^*,\omega_D)$.
Then
\[
 e\succ_t f=(1+t)f,\qquad F\prec_t F=(1-t)E,
\]
\[
 f\vee_t f=(1+t)e,\qquad F\vee_t F=-(1-t)E.
\]
The original and deformed product triples are linearly independent:
if the deformation were $c$ times the original triple, the first
nonzero product would force $c=1$ and the second would force $c=-1$.

The family has the fixed factorizable tensor
\[
 r_D=e\o e^\vee+f\o f^\vee+e^*\o E+f^*\o F,
\]
whose coproducts are determined by
\meqref{eq:nijenhuis-coproduct-pencil} and
\meqref{eq:nijenhuis-coproduct-explicit}.
\end{ex}

\subsection*{Acknowledgment}
The work is supported by Natural Science Foundation of China (No.12271292, 12471036).
\subsection*{Author declarations}
The authors have no conflicts of interest to disclose.
\subsection*{Data availability statement}
No data were used for the research described in this article.

\end{document}